\documentclass[12pt,a4paper]{amsart}

\usepackage{amssymb}
\usepackage{mathabx}
\usepackage{bbm}
\usepackage{amsthm}
\usepackage{dsfont}
\usepackage{amsmath}
\usepackage{color,graphics,srcltx}
\usepackage{bm}
\usepackage{easybmat}
\usepackage{multirow,bigdelim}
\usepackage{enumerate}
\usepackage{graphicx}
\usepackage{amsthm}

\newtheorem{proposition}{Proposition}
\newtheorem{lemma}[proposition]{Lemma}

\newtheorem{theorem}[proposition]{Theorem}

\theoremstyle{definition}
\newtheorem{example}[proposition]{Example}
\newtheorem{definition}[proposition]{Definition}

\newcommand{\RR}{\mathbb{R}}

\def\cQ{\mathcal{Q}}

\definecolor{damjana}{rgb}{.8,.2,.2}

\def\RR{\mathds R}

\def\II{\mathds I}

\def\id{\mathrm{id}}

\def\CC{\mathcal C}

\begin{document}

\title{On the exact region determined by Spearman's rho and Spearman's footrule}

\author[D. Kokol Bukov\v sek]{Damjana Kokol Bukov\v{s}ek}
\address{School of Economics and Business, University of Ljubljana, and Institute of Mathematics, Physics and Mechanics, Ljubljana, Slovenia}
\email{damjana.kokol.bukovsek@ef.uni-lj.si}

\author[N. Stopar]{Nik Stopar}
\address{Faculty of Electrical Engineering, University of Ljubljana, and Institute of Mathematics, Physics and Mechanics, Ljubljana, Slovenia}
\email{nik.stopar@fe.uni-lj.si}

\begin{abstract}
We determine the lower bound for possible values of Spearman's rho of a bivariate copula given that the value of its Spearman's footrule is known and show that this bound is always attained.
We also give an estimate for the exact upper bound and prove that the estimate is attained for some but not all values of Spearman's footrule. Nevertheless, we show that the estimate is quite tight.
\end{abstract}

\thanks{Both authors acknowledge financial support from the Slovenian Research Agency (research core funding No. P1-0222).}

\maketitle

\section{ Introduction }

Intuitively, two continuous random variables $X$ and $Y$ are in concordance when large values of $X$ occur simultaneously with large values of $Y$. More precisely, two realisations $(x_1,y_1)$ and $(x_2,y_2)$ of the random vector $(X,Y)$ are concordant when $(x_2-x_1)(y_2-y_1) >0$ and they are discordant when $(x_2-x_1)(y_2-y_1) <0$. We can measure the concordance of a pair of random variables $(X,Y)$ in various ways, see \cite{Scar}. A concordance measure is often a better way to model dependence than Pearson's correlation coefficient since it is invariant with respect to monotone increasing transformations of the random variables. Because of this invariance, the concordance of a random vector $(X,Y)$ is uniquely determined by its copula, which is given by $$C(u,v)=H(F^{(-1)}(u), G^{(-1)}(v)),$$
where $H$ is a joint distribution function of $(X,Y)$, $F,G$ are univariate distribution functions of random variables $X$ and $Y$, respectively, and $F^{(-1)}$, $G^{(-1)}$ are their generalised inverses.

Due to their importance in statistical analysis, which is based on their connection to measures for the degree of association between two random variables, concordance measures have been studied intensively since their introduction.
Recent references for bivariate concordance measures include \cite{EdTa,FrNe,FuSch,KoBuKoMoOm3,Lieb,NQRU,NQRU2} and their multivariate generalizations were studied in \cite{BeDoUF2,DuFu,Tayl,UbFl}, to name just a few.
Given their widespread use in a variety of practical applications, it is natural to compare different concordance measures in terms of the values that they can attain. In particular, if a value of one measure is known, we may ask what are the possible values of the other measures.
Here we only give a brief overview of the known results regarding this question, and for explicit formulas we refer the reader to the papers referenced below or to \cite{KoBuMo}, where all the formulas are collected in one place.
The investigation of the above question was started by Daniels~\cite{Dani} and Durbin and Stuart~\cite{DuSt}, who compared Spearman's rho and Kendall's tau and gave some estimates for the values of the two measures.
The exact region of all possible pairs of values $(\tau(C),\rho(C))$, $C \in \CC$, was only determined recently in \cite{ScPaTr}. 
The regions determined by Blomqvist's beta and the other three concordance measures (Spearman's rho, Kendall's tau, and Gini's gamma) are given in \cite{Nels} as an exercise for the reader, while the region determined by Blomqvist's beta and Spearman's footrule was given in \cite{KoBuKoMoOm3}.  The region determined by Spearman's footrule and Gini's gamma was given in \cite{KoBuMo}.
As far as we know, the regions determined by other pairs of (weak) concordance measures are not yet known.

In this paper we investigate the relation between Spearman's rho and Spearman's footrule of a bivariate copula.
We  determine the exact lower bound for the value of Spearman's rho if the value of Spearman's footrule is known. The determination of the exact upper bound seems to be quite difficult, but we are able to give a tight estimate for it.

The paper is structured as follows. In Section~\ref{sec:prelim} we give some basic definitions that will be used throughout the paper.
In Section~\ref{sec:symmetric} we define doubly symmetric shuffles of $M$ and prove that any doubly symmetric copula can be approximated by a doubly symmetric shuffle of $M$.
Sections~\ref{sec:special} and \ref{sec:general} are devoted to determining the exact lower bound for Spearman's rho in terms of Spearman's footrule and the corresponding upper bound is considered in Section~\ref{sec:upper}.
We give an estimate for the exact upper bound and prove that the estimate is attained for some but not all values of Spearman's footrule. Nevertheless, we show that the estimate is quite tight.
In Section~\ref{sec:region} we estimate the similarity measure between Spearman's footrule and Spearman's rho.

\section{Preliminaries on concordance measures}\label{sec:prelim}

Let $\II=[0,1]$ be the unit interval and $B=[u_1,u_2]\times [v_1,v_2]$ a rectangle contained in $\II^2$ with $u_1 \le u_2$ and $v_1 \le v_2$.
Given a real function $H \colon \II^2 \to \RR$, we define the $H$-volume of rectangle $B$ by
$V_H(B)=H(u_2,v_2)-H(u_2,v_1)-H(u_1,v_2)+H(u_1,v_1)$.
A {\it bivariate copula} is a function $C:\II^2 \to \II$ with the following properties:
\begin{enumerate}[(i)]
    \item $C(0,v)=C(u,0)=0$ for all $u,v \in \II$ ($C$ is {\it grounded}),
    \item $C(u,1)=u$ and $C(1,v)=v$ for all $u,v \in \II$ ($C$ has uniform marginals), and
    \item $V_H(B) \geq 0$ for every rectangle $B \subseteq \II^2$ ($C$ is {\it $2-$increasing}).
\end{enumerate}
Since we will only be dealing with the bivariate setting, we will often omit the adjective and simply call such functions copulas.
The set of all bivariate copulas will be denoted by $\CC$. It is well known that this set is compact in the sup norm.

Let us introduce some standard transformations that are naturally defined on $\CC$ and are induced by reflections of the unit square $\II^2$. We denote by $C^t$ the transpose of the copula $C$, i.e., $C^t(u,v)= C(v,u)$, which is induced by the reflection over the main diagonal.
A copula $C$ that satisfies the condition $C=C^t$ is called \emph{symmetric}.
The two reflections $\sigma_1 \colon (u,v) \mapsto (1-u,v)$ and $\sigma_2 \colon (u,v) \mapsto (u,1-v)$ induce reflections $C^{\sigma_1}$ and $C^{\sigma_2}$ of the copula $C$, which are defined by $C^{\sigma_1}(u,v) =v-C(1-u,v)$ and $C^{\sigma_2}(u,v)=u-C(u,1-v)$ (see \cite[{\S}1.7.3]{DuSe}), and are again copulas.
If we apply both reflections to $C$, we obtain the survival copula of $C$, which we denote by $\widehat{C}=\left(C^{\sigma_1}\right)^{\sigma_2}=\left(C^{\sigma_2}\right)^{\sigma_1}$.
It is induced by the reflection $(u,v) \mapsto (1-u,1-v)$ and is given by $\widehat{C}(u,v)=u+v-1+C(1-u,1-v)$.

Given two copulas $C$ and $D$ we denote $C\le D$ if $C(u,v)\leqslant D(u,v)$ for all $(u,v)\in\II^2$. This is the so-called \textit{pointwise order} of copulas.
The set $\CC$ equipped with the pointwise order is a partially ordered set, but not a lattice \cite[Theorem 2.1]{NeUbF2}. For any copula $C$ we have $W \le C \le M$, where $W(u,v)=\max\{0,u+v-1\}$ and $M(u,v)=\min\{u,v\}$ are the lower and upper \textit{Fr\'echet-Hoeffding bounds} for the set of all copulas.

Formal axioms for a concordance measure were introduced by Scarsini \cite{Scar} (see also \cite{DoUbFl} and \cite{Tayl2} for multivariate versions).
A mapping $\kappa:\CC\to [-1,1]$ is called a \textit{concordance measure} if it satisfies the following properties (see \cite[Definition 2.4.7]{DuSe}):
\begin{enumerate}[(C1)]
\item $\kappa(C)=\kappa(C^t)$ for every $C\in\CC$.
\item $\kappa(C)\leqslant\kappa(D)$ when $C\leqslant D$.  \label{monotone}
\item $\kappa(M)=1$.
\item $\kappa(C^{\sigma_1})=\kappa(C^{\sigma_2})=-\kappa(C)$.
\item If a sequence of copulas $C_n$, $n\in\mathbb{N}$, converges uniformly to $C\in\CC$, then $\displaystyle\lim_{n\to\infty}\kappa(C_n)=\kappa(C)$.
\end{enumerate}
If a sequence of copulas converges pointwise to a function $C$, then $C$ is a copula and the sequence converges uniformly to $C$ (see \cite{DuSe} for details).
Hence, in axiom (C5) we may replace the uniform convergence requirement with the pointwise convergence condition without loss of generality.
Any concordance measure automatically satisfies also the following additional properties, which are sometimes stated as part of the definition, but are actually consequences of conditions (C1)-(C5) (see \cite[\S{3}]{KoBuKoMoOm2} for more details):
\begin{enumerate}[(C1)]\setcounter{enumi}{5}
\item $\kappa(\Pi)=0$, where $\Pi$ is the independence copula $\Pi(u,v)=uv$. \label{kappa_Pi}
\item $\kappa(W)=-1$.\label{kappaW}
\item $\kappa(C)=\kappa(\widehat{C})$ for every $C\in\CC$.
\end{enumerate}

Many of the most important bivariate concordance measures can be expressed with the so called \emph{concordance function} $\cQ$, introduced by Kruskal \cite{Krus} (see also \cite{BeDoUF,EdMiTa,EdMiTa2}), which measures the difference between the probabilities of concordance and discordance of two pairs of random variables.
If $(X_1,Y_1)$ and $(X_2,Y_2)$ are pairs of continuous random variables $X_1$, $X_2$, $Y_1$, and $Y_2$,
then the concordance function of random vectors $(X_1,Y_1)$ and $(X_2,Y_2)$ depends only on the corresponding copulas $C_1$ and $C_2$ and is given by (see \cite[Theorem 5.1.1]{Nels})
\begin{equation}\label{concordance} \cQ=\cQ(C_1,C_2)= 4 \int_{\II ^2} C_2(u,v)dC_1(u,v) -1.
\end{equation}
It turns out that the concordance function is symmetric in its arguments, i.e., $\cQ(C_1,C_2)=\cQ(C_2,C_1)$, and has several other useful properties, see \cite[Corollary 5.1.2]{Nels} and \cite[\S{3}]{KoBuKoMoOm2}.

The four most commonly used concordance measures of a copula $C$ are Spearman's rho, Kendall's tau, Gini's gamma, and Blomqvist's beta. The first three can be defined in terms of the concordance function $\cQ$.
The \textit{Spearman's rho} is defined by
\begin{equation}\label{rho}
\rho(C)=3\,\cQ(C,\Pi)=12\int_{\II ^2} C(u,v) dudv - 3,
\end{equation}
\textit{Kendall's tau} by
\begin{equation}\label{tau}
\tau(C)=\cQ(C,C)= 4\int_{\II ^2} C(u,v) dC(u,v) - 1,
\end{equation}
and \textit{Gini's gamma} by
\begin{equation}\label{gamma}
\gamma(C)=\cQ(C,M)+\cQ(C,W) = 4\int_0^1 C(u,u) du + 4\int_0^1 C(u,1-u) du - 2.
\end{equation}
On the other hand, \textit{Blomqvist's beta} is defined by
\begin{equation}\label{beta}
\beta(C)=4\,C\left(\frac12,\frac12\right)-1.
\end{equation}
We refer the reader to \cite[{\S}2.4]{DuSe} and \cite[Ch. 5]{Nels} for more details on these measures.

In 2014 Liebscher \cite{Lieb} considered measures that are slightly more general than concordance measures, in paritcular, if we replace property (C4) with property (C6) in the definition of a concordance measure, we get what Liebscher calls \textit{weak concordance measure}.
The most important example of a weak concordance measure is the \textit{Spearman's footrule} defined by
\begin{equation}\label{phi}
\phi(C)=\frac12\left(3\cQ(C,M)-1\right)=6\int_0^1 C(u,u) du - 2.
\end{equation}
While the range of any concordance measure is the interval $[-1,1]$, the range of a weak concordance measure may be different. For example, the range of Spearman's footrule is the interval $\left[-\frac12,1\right]$ (see \cite[\S4]{UbFl}).
Note that Spearman's rho, Kendall's tau, Gini's gamma and Blomqvist's beta are concordance measures, hence, they satisfy conditions (C1)-(C8). On the other hand, Spearman's footrule only satisfies conditions (C1)-(C3), (C5)-(C6) and (C8).

All five (weak) concordance measures mentioned above are well established in statistical literature and their statistical meaning has been investigated in detail. See e.g. \cite{KGJT,NSB,SMB,Wys} for Spearman's rho, \cite{FuSch2,JaMa,KGJT,Wys} for Kendall's tau, \cite{CoNi,GeNeGh,Nels1998} for Gini's gamma, \cite{UbFl} for Blomqvist's beta and \cite{DiaGra,GeNeGh,UbFl,SSQ} for Spearman's footrule.

The known exact regions determined by pairs of (weak) concordance measures are shown in  Figures~\ref{fig:tau-ro} and \ref{fig:beta}.

\begin{figure}[!ht]
    \centering
    \includegraphics[height=5cm]{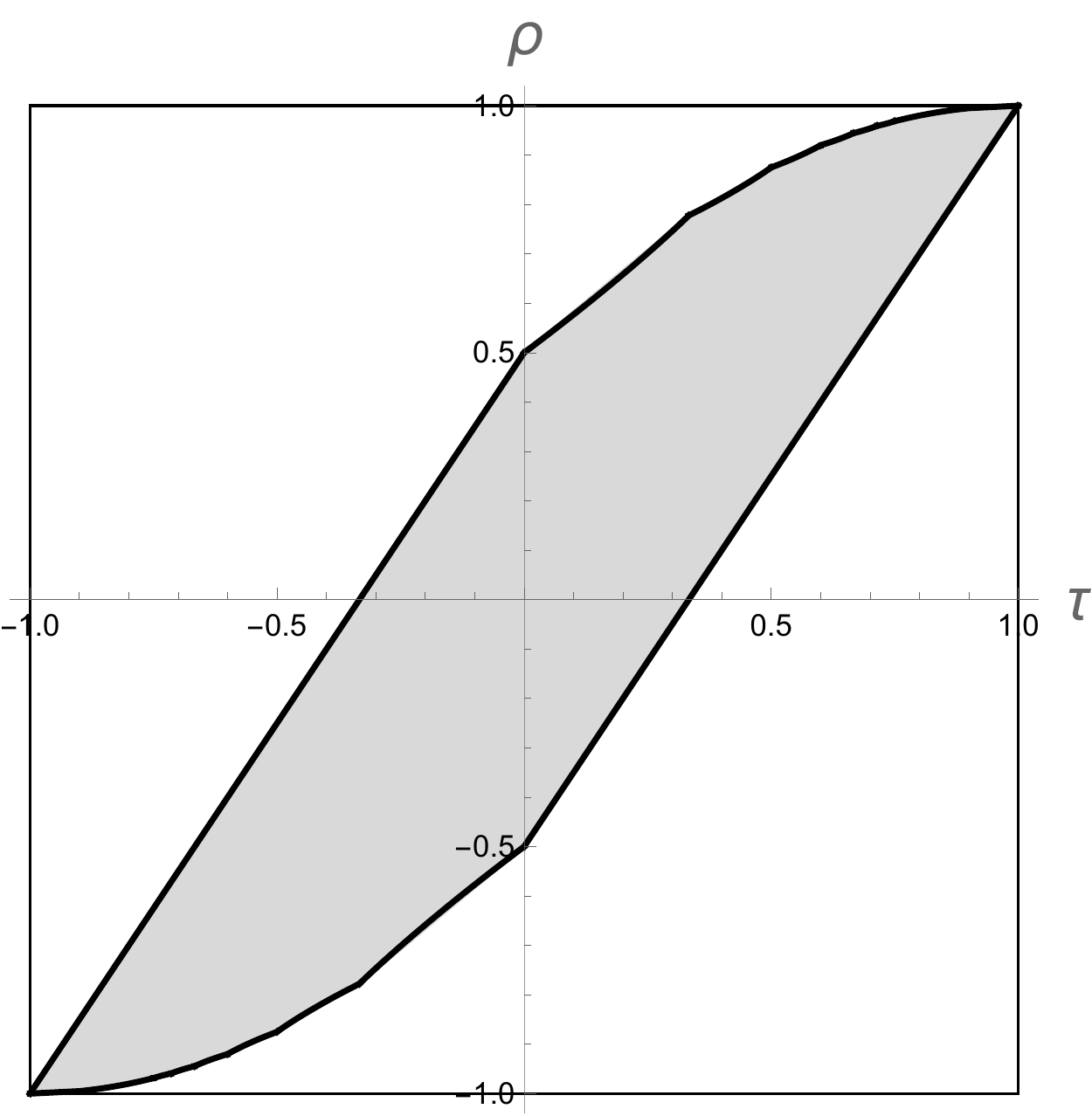} \qquad
    \includegraphics[height=5cm]{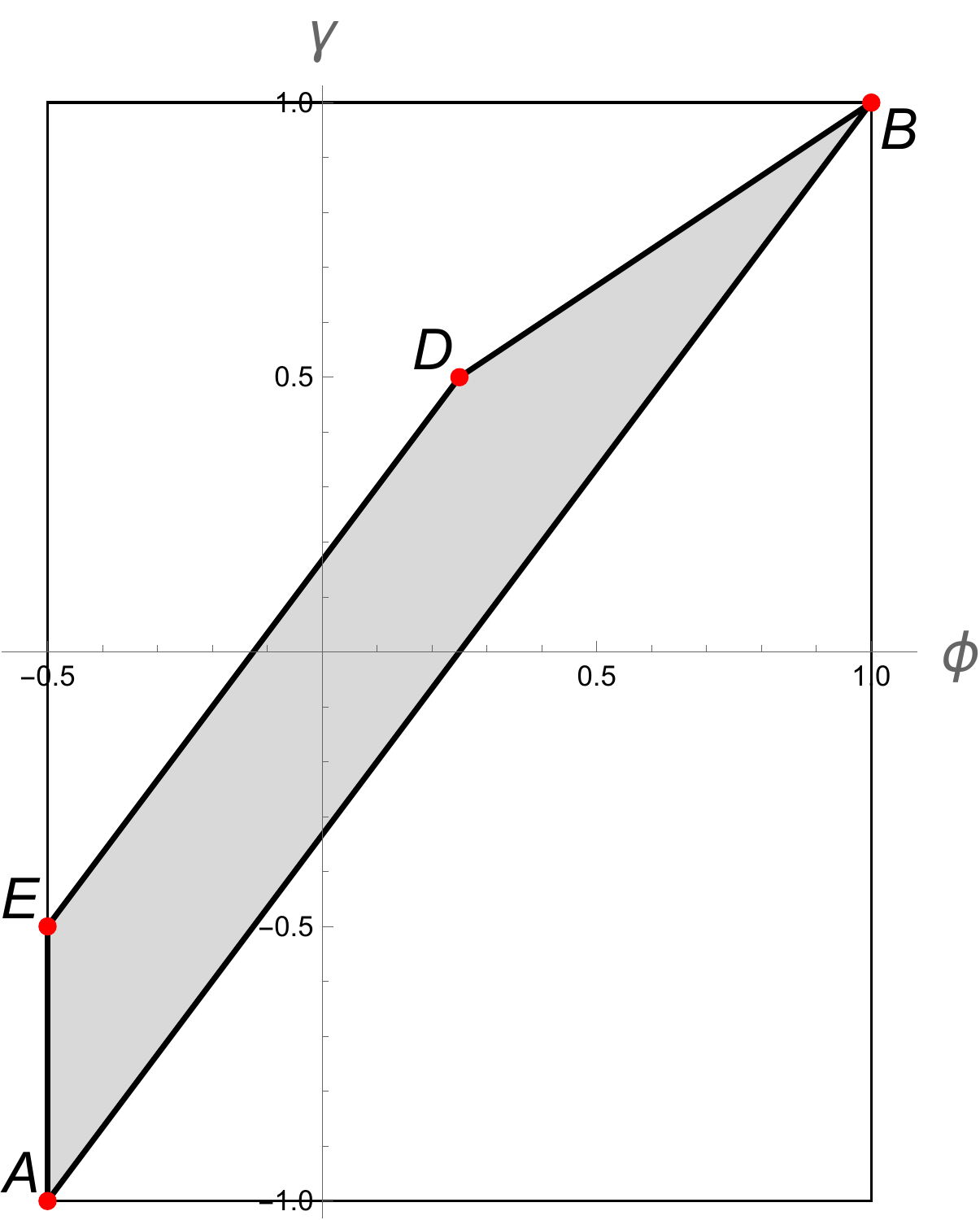}
    \caption{The exact region determined by Kendall’s tau and Spearman’s rho (left) and by Spearman’s footrule and Gini’s gamma (right).}
    \label{fig:tau-ro}
\end{figure}

\begin{figure}[!ht]
    \centering
    \includegraphics[width=5cm]{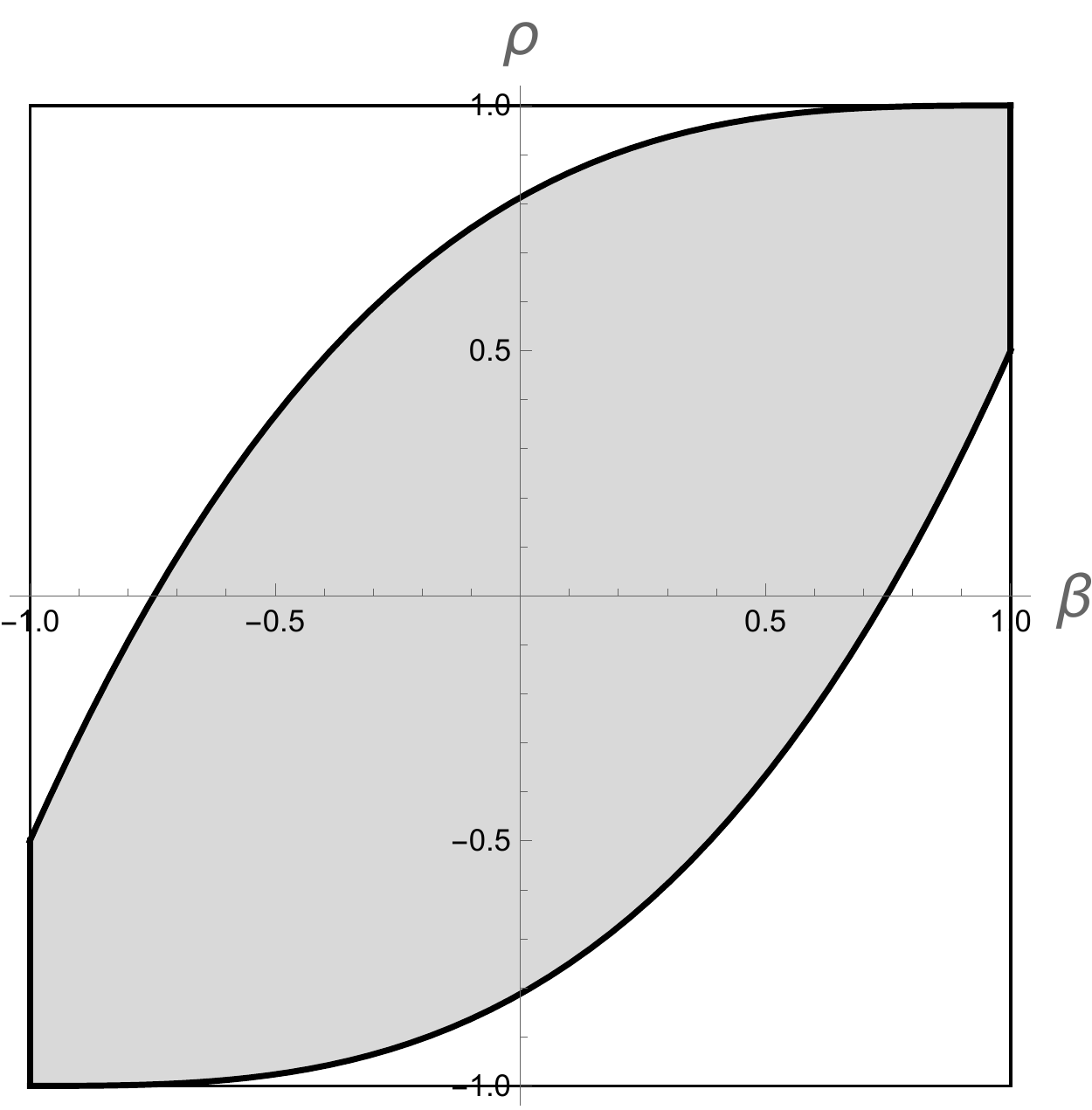} \hspace{4mm}
    \includegraphics[width=5cm]{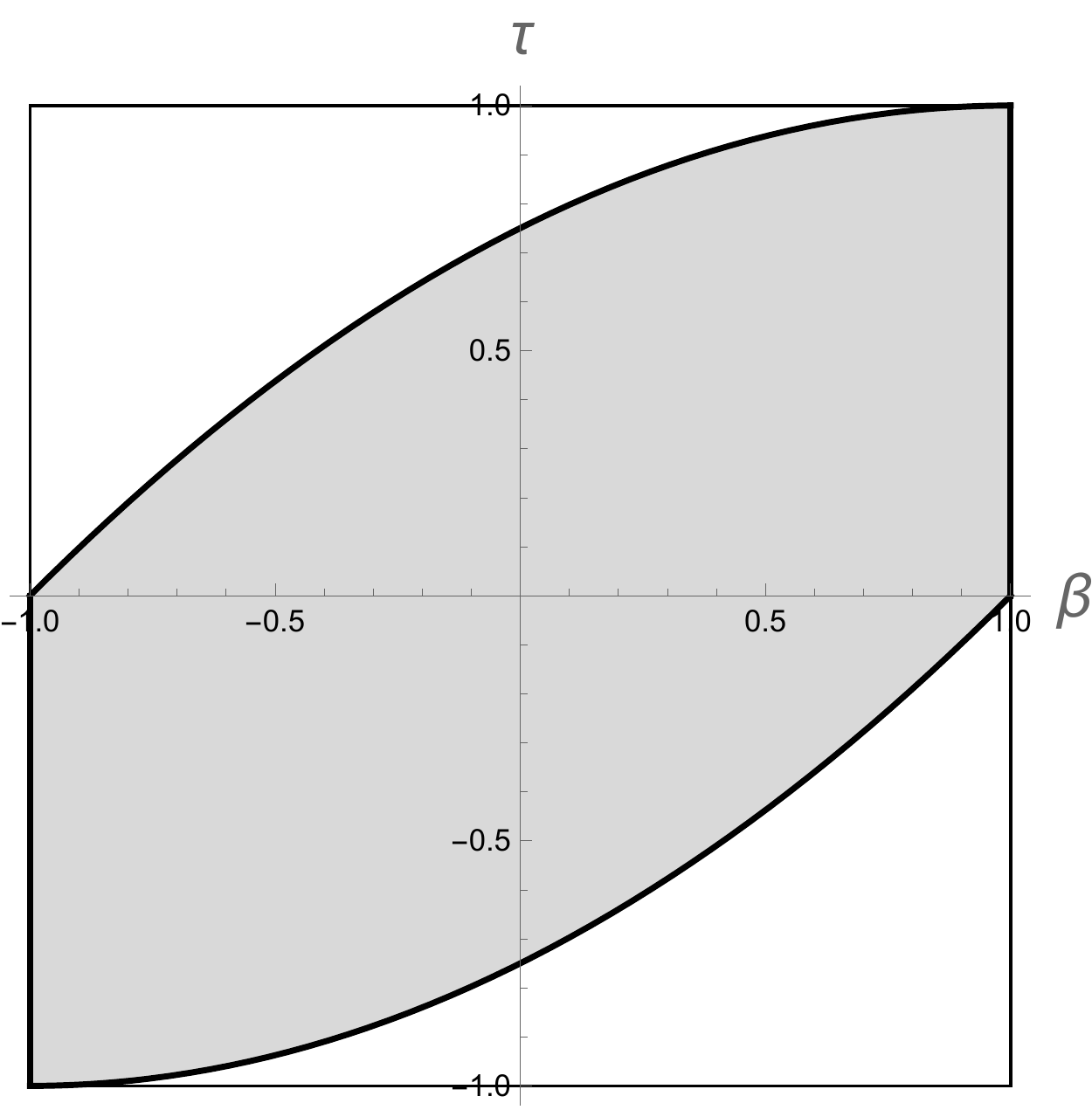}
    \vspace{2mm}
    
    \includegraphics[width=5cm]{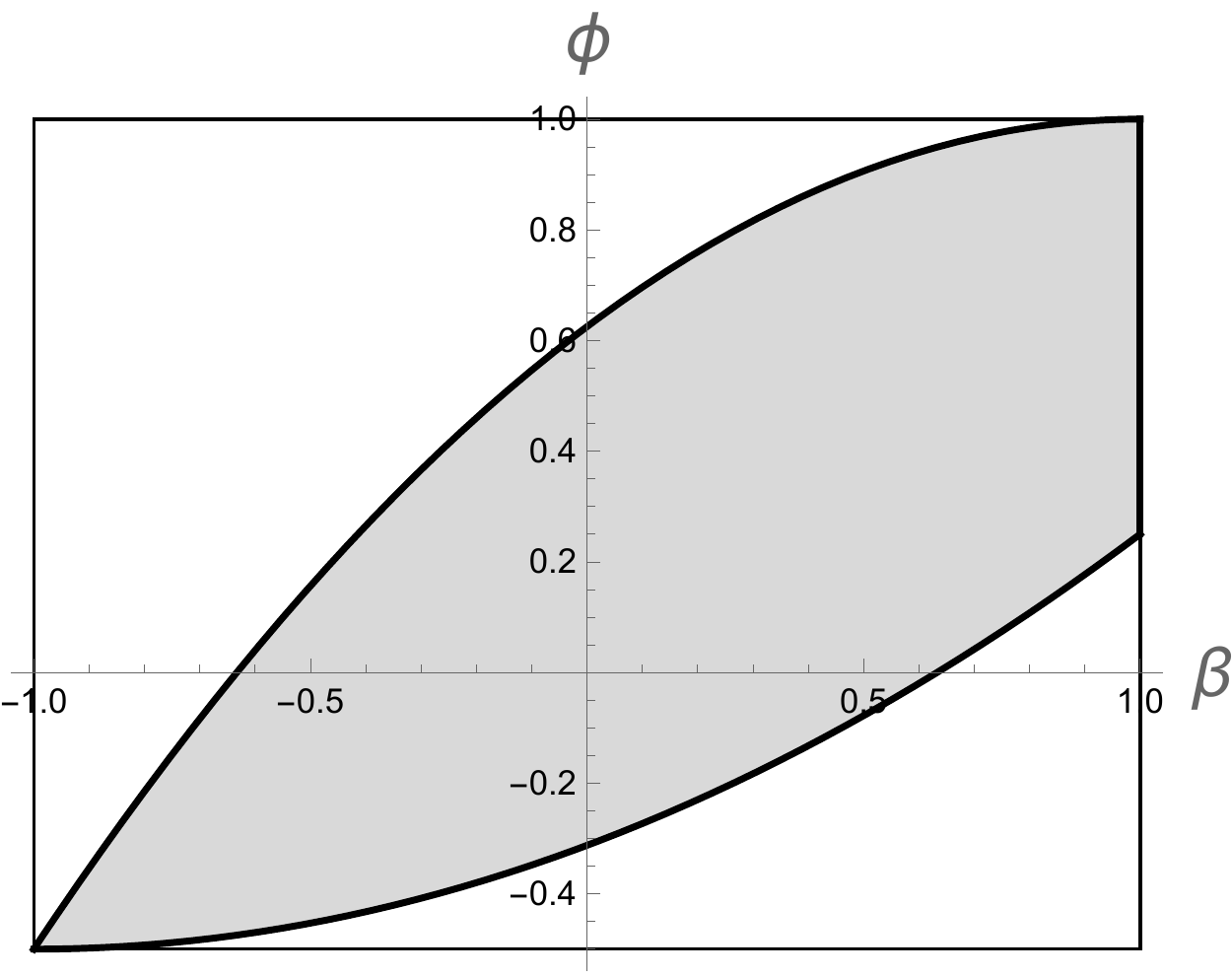} \hspace{4mm}
    \includegraphics[width=5cm]{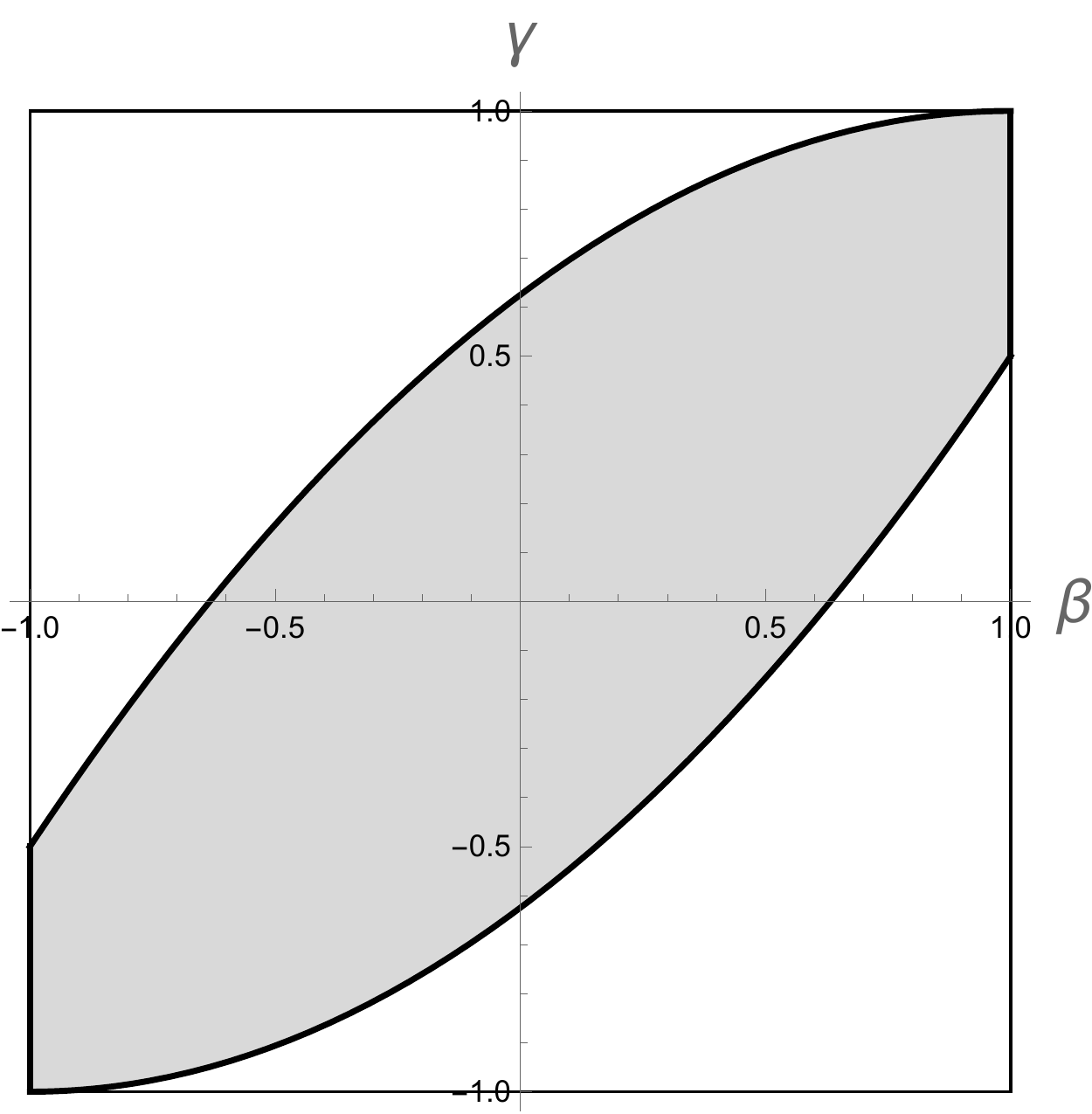}
    \caption{The exact regions determined by Blomqvist’s beta and Spearman’s rho, Kendall’s tau, Spearman’s footrule, and Gini’s gamma, respectively.}
    \label{fig:beta}
\end{figure}

\section{Doubly symmetric shuffles}\label{sec:symmetric}

It is well known that any copula can be approximated arbitrarily well in the sup norm by a shuffle of $M$.
In this section we investigate approximations of doubly symmetric copulas with doubly symmetric shuffles of $M$, defined below.

\begin{definition}
A copula $C$ is called \emph{doubly symmetric} if $C=C^t=\widehat{C}$.
\end{definition}

Note that a copula is doubly symmetric if and only if the distribution of its mass in the unit square is symmetric with respect to both the main and the opposite diagonal.
The reflection with respect to the main diagonal is given by $(u,v) \mapsto (v,u)$ while the reflection with respect to the opposite diagonal is given by $(u,v) \mapsto (1-v,1-u)$.

A shuffle of $M$
$$ C = M(n, J, \pi, \omega)$$
is determined by a positive integer $n$, a partition
$J=\{J_1,J_2,\ldots, J_n\}$ of the interval $\II$ into $n$ pieces, where $J_i=[u_{i-1},u_i]$ and $0 =u_0 \le u_1 \le u_2 \le \ldots \le u_{n-1} \le u_n=1$, shortly written as $(n-1)$-tuple of splitting points $J=(u_1, u_2, \ldots, u_{n-1})$,
a permutation $\pi \in S_n$, written as $n$-tuple of images $\pi = (\pi(1), \pi(2), \ldots, \pi(n))$, and a mapping $\omega: \{1, 2, \ldots, n\} \to \{-1, 1\}$, written as $n$-tuple of images $\omega = (\omega(1), \omega(2), \ldots, \omega(n))$. The mass of $C$ is concentrated on squares $J_i \times  [v_{\pi(i)-1} \times v_{\pi(i)}]$ where $0= v_0 \le v_1 \le v_2\le \ldots \le v_{n-1} \le v_n = 1$. For more details see \cite[\S3.2.3]{Nels}. Notice that we allow some of the intervals in the partition $J$ to be singletons. We can now define doubly symmetric shuffles of $M$.

\begin{definition}
We will say that a shuffle $C=M(n,J,\pi,\omega)$ of $M$ with $J=(u_1,u_2,\ldots,u_{n-1})$, $u_0=0$, $u_n=1$, is a \emph{doubly symmetric shuffle} if the following properties hold
\begin{enumerate}[(i)]
    \item $n$ is even,
    \item $\pi^2=\id$ and $\pi(n-i+1)=n-\pi(i)+1$ for all $i=1,2,\ldots,n$,
    \item $\omega(i)=\omega(\pi(i))=\omega(n-i+1)$ for all $i=1,2,\ldots,n$,
    \item $u_i-u_{i-1}=u_{\pi(i)}-u_{\pi(i)-1}=u_{n-i+1}-u_{n-i}$ for all $i=1,2,\ldots,n$.
\end{enumerate}
\end{definition}

In the following lemma we give some properties of doubly symmetric shuffles.

\begin{lemma}\label{lem:shuffle_prop}
If $C$ is a doubly symmetric shuffle of $M$ then
\begin{enumerate}[(a)]
\item $C$ is a doubly symmetric copula,
\item $u_{n-i}=1-u_i$ for all $i=1,2,\ldots,n$, in particular, $u_{\frac{n}{2}}=\frac{1}{2}$,
\item all the mass of $C$ is concentrated on the squares $J_i \times J_{\pi(i)}=[u_{i-1},u_i] \times [u_{\pi(i)-1},u_{\pi(i)}]$, $i=1,2,\ldots,n$, and \begin{align*}
V_C(J_i \times J_{\pi(i)}) &=V_C(J_{\pi(i)} \times J_{i})=V_C(J_{n-i+1} \times J_{\pi(n-i+1)})\\
&=V_C(J_{\pi(n-i+1)} \times J_{n-i+1})=u_i-u_{i-1}.
\end{align*}
\end{enumerate}
\end{lemma}

\begin{proof}
Let us first prove (b). Note that
$$u_{n-i}=\sum_{k=1}^{n-i} (u_{k}-u_{k-1})=\sum_{k=1}^{n-i} (u_{n-k+1}-u_{n-k})$$
by (iv). By introducing a new index $j=n-k+1$ we get
$$u_{n-i}=\sum_{j=i+1}^{n} (u_{j}-u_{j-1})=u_n-u_i=1-u_i.$$
The equality $u_{\frac{n}{2}}=\frac12$ follows easily by taking $i=\frac{n}{2}$.

To prove (c) let $v_0=0$ and $v_n=1$ and let $v_1,v_2,\ldots,v_{n-1}$ be splitting points on the $v$-axis such that the mass of $C$ is concentrated on squares
$J_i \times J_{\pi(i)}^{\pi^{-1}}=[u_{i-1},u_i] \times [v_{\pi(i)-1} \times v_{\pi(i)}]$. This means that $v_{\pi(i)}-v_{\pi(i)-1}=u_i-u_{i-1}$ for all $i$, so that
$v_j=v_{j-1}+u_{\pi^{-1}(j)}-u_{\pi^{-1}(j)-1}$. By induction it follows that $$v_j=\sum_{k=1}^j (u_{\pi^{-1}(k)}-u_{\pi^{-1}(k)-1}).$$
Since $C$ is doubly symmetric shuffle of $M$ we have $\pi^{-1}=\pi$, hence
$$v_j=\sum_{k=1}^j (u_{\pi(k)}-u_{\pi(k)-1})=\sum_{k=1}^j (u_{k}-u_{k-1})=u_j,$$
for all $j=1,2,\ldots,n-1$, where the second equation follows from item (iv) of the definition of a doubly symmetric shuffle of $M$. Clearly also $v_0=u_0$ and $v_n=u_n$.
This implies that $J_{\pi(i)}^{\pi^{-1}}=[v_{\pi(i)-1} \times v_{\pi(i)}]=[u_{\pi(i)-1} \times u_{\pi(i)}]=J_{\pi(i)}$ for all $i=1,2,\ldots,n$ and $V_C(J_i \times J_{\pi(i)})=u_i-u_{i-1}$.
Now $V_C(J_{\pi(i)} \times J_{i})=u_{\pi(i)}-u_{\pi(i)-1}=u_i-u_{i-1}$ follows from the above and item (iv). Similarly, we prove for the other two volumes which proves (c).
Thus, to prove (a) it suffices to show that 
the set of squares $J_{i} \times J_{\pi(i)}$ with $\omega(i)=1$ is invariant under reflection with respect to both diagonals and the same holds for the set of squares $J_{i} \times J_{\pi(i)}$ with $\omega(i)=-1$.
The reflection with respect to the main diagonal reflects the square $J_{i} \times J_{\pi(i)}=[u_{i-1},u_{i}] \times [u_{\pi(i)-1},u_{\pi(i)}]$ onto $ [u_{\pi(i)-1},u_{\pi(i)}] \times [u_{i-1},u_i]=J_{\pi(i)} \times J_{i}=J_{i'} \times J_{\pi(i')}$, where $i'=\pi(i)$ by (ii).
Since $\omega(i)=\omega(\pi(i))$ by (iii), the original and the reflected square have the same value of $\omega$.
The reflection with respect to the opposite diagonal reflects the square $J_{i} \times J_{\pi(i)}=[u_{i-1},u_i] \times [u_{\pi(i)-1},u_{\pi(i)}]$ onto $ [1-u_{\pi(i)},1-u_{\pi(i)-1}] \times [1-u_i,1-u_{i-1}]=[u_{n-\pi(i)},u_{n-\pi(i)+1}]\times[u_{n-i},u_{n-i+1}]=J_{n-\pi(i)+1} \times J_{n-i+1}$ by (b), and $J_{n-\pi(i)+1} \times J_{n-i+1}=J_{\pi(n-i+1)} \times J_{n-i+1}=J_{i''} \times J_{\pi(i'')}$ by (ii), where $i''=\pi(n-i+1)$.
Since $\omega(i)=\omega(n-i+1)=\omega(\pi(n-i+1))=\omega(i'')$ by (iii), the original and the reflected square have the same value of $\omega$. This finishes the proof of (a).
\end{proof}

Note that a doubly symmetric copula $C$ which is also a shuffle $M(n,J,\pi,\omega)$ is not necessarily a doubly symmetric shuffle. But there exists a doubly symmetric shuffle $M(n',J',\pi',\omega')$ such that $C=M(n',J',\pi',\omega')$. For example, if $\frac12$ is not one of the splitting points, we can add it and adjust $\pi$ and $\omega$ accordingly.

In next lemma we prove that doubly symmetric copulas can be approximated by doubly symmetric shuffles of $M$.
In fact, we show that the original construction by Mikusi\'{n}ski et al. \cite{Miku} (c.f. also \cite[\S 3.2.3]{Nels}) of a shuffle of $M$ that approximates a copula $C$ produces a doubly symmetric shuffle of $M$ whenever $C$ is a doubly symmetric copula.
Recall that a shuffle of $M$ is called \emph{straight} if the mapping $\omega$ has all values equal to $1$.

\begin{lemma}\label{lem:approx}
For any doubly symmetric copula $C$ and any $\varepsilon>0$ there exists a straight doubly symmetric shuffle of $M$, which we denote by $C'$, such that 
$$\sup_{(u,v) \in \II^2}|C(u,v)-C'(u,v)| <\varepsilon.$$
\end{lemma}

\begin{proof}
Suppose $C$ is a doubly symmetric copula.
Let $m$ be an even positive integer and $n=m^2$, so that $n$ is even as well.
Let $J=\{J_1,J_2,\ldots,J_n\}$ be a partition of $\II$ such that the intervals $J_1,J_2,\ldots,J_n$ are ordered from left to right and the length of $J_i$ is equal to $w_i=V_C([\tfrac{k-1}{m},\tfrac{k}{m}]\times[\tfrac{j-1}{m},\tfrac{j}{m}])$, where $i=m(j-1)+k$, $j=1,2,\ldots,m$, $k=1,2,\ldots,m$,
and let $\pi \in S_n$ be a permutation given by $\pi(m(j-1)+k)=m(k-1)+j$ for all $j,k \in \{1,2,\ldots,m\}$.
Finally, let $\omega$ be constantly equal to $1$.
We claim that $C_m=M(n,J,\pi,\omega)$ is a doubly symmetric shuffle.
For the rest of the proof we let $i=m(j-1)+k$, where $j,k \in \{1,2,\ldots,m\}$. Clearly, $\pi^2=\id$. Furthermore,
\begin{align*}
\pi(n-i+1)&=\pi(m^2-m(j-1)-k+1)=\pi(m(m-j)+(m-k+1))\\
&=m(m-k)+(m-j+1)=m^2-m(k-1)-j+1=n-\pi(i)+1
\end{align*}
for all $i=1,2,\ldots,n$.
Since copula $C$ is doubly symmetric, its mass is symmetric with respect to both diagonals.
The symmetry with respect to the main diagonal implies
\begin{equation}\label{eq:vi_1}
w_i=V_C([\tfrac{k-1}{m},\tfrac{k}{m}]\times[\tfrac{j-1}{m},\tfrac{j}{m}])=V_C([\tfrac{j-1}{m},\tfrac{j}{m}]\times[\tfrac{k-1}{m},\tfrac{k}{m}])=w_{m(k-1)+j}=w_{\pi(i)}
\end{equation}
for all $i \in \{1,2,\ldots,n\}$, while the symmetry with respect to the opposite diagonal implies
\begin{equation}\label{eq:vi_2}
\begin{aligned}
w_i &=V_C([\tfrac{k-1}{m},\tfrac{k}{m}]\times[\tfrac{j-1}{m},\tfrac{j}{m}])=V_C([1-\tfrac{j}{m},1-\tfrac{j-1}{m}]\times[1-\tfrac{k}{m},1-\tfrac{k-1}{m}]) \\
&=V_C([\tfrac{m-j}{m},\tfrac{m-j+1}{m}]\times[\tfrac{m-k}{m},\tfrac{m-k+1}{m}])=w_{m(m-k)+m-j+1}=w_{m^2-m(k-1)-j+1}=w_{n-\pi(i)+1}
\end{aligned}
\end{equation}
for all $i \in \{1,2,\ldots,n\}$.
By definition of partition $J$, its splitting points are
$u_0=0$ and $u_i=\sum_{k=1}^{i}w_k$ for all $i=1,2,\ldots,n$, where $u_n=\sum_{k=1}^{n}w_k=V_C(\II \times \II)=1$.
Together with equalities \eqref{eq:vi_1} and \eqref{eq:vi_2} this implies
$$u_i-u_{i-1}=w_i=w_{\pi(i)}=u_{\pi(i)}-u_{\pi(i)-1}$$
and
$$u_i-u_{i-1}=w_i=w_{\pi(i)}=w_{n-\pi^2(i)+1}=w_{n-i+1}=u_{n-i+1}-u_{n-i}.$$
We have thus shown that $C_m$ is a straight doubly symmetric shuffle of $M$.
To finish the proof we note that as $m$ tends to $\infty$, the copula $C_m$ converges uniformly to $C$ as proved in \cite[Theorem~3.1]{Miku} (cf. also \cite[\S 3.2.3]{Nels}).
\end{proof}

\section{Lower bound - special case}\label{sec:special}

In the following two sections we prove a lower bound for the value of Spearman's rho for any copula with a given value of Spearman's footrule.
In this section we consider copulas which have all the mass concentrated on the main and opposite diagonal. In the next section we will reduce the general case to this special case.

Copulas which have all mass concentrated on the two diagonals correspond to uniformly distributed random variables $U$ and $V$ on interval $\II$ with the property $P(U=V) + P(U=1-V) =1$.
As we will see, any such copula is completely determined by its diagonal $\delta_C(u) = C(u,u)$ and it is doubly symmetric, so it is easier to tackle. 

\begin{proposition}\label{prop:x}
  Let $U$ and $V$ be uniformly distributed random variables on interval $\II$ with the property $P(U=V) + P(U=1-V) =1$ and let $C\in\CC$ be their copula. Let $\delta_C(u) = C(u,u)$ be the diagonal of $C$ and define 
  $$\alpha_C(u)= \int_0^u \delta_C(t)dt .$$ Let $\phi(C)=p$ for some $p\in[-\frac12,1]$ and $u_0=\frac12(1-\frac{1}{\sqrt{3}}\sqrt{2p+1}) \in [0,\frac12]$. Then the following holds: 
  \begin{enumerate}
      \item[(a)] $$C(u,v) = \begin{cases}
            \delta_C(u); & 0 \le u \le \frac12, u \le v \le 1-u, \\
            \delta_C(v); & 0 \le v \le \frac12, v \le u \le 1-v, \\
            \delta_C(u)+v-u; & \frac12 \le u \le 1, 1-u \le v \le u, \\
            \delta_C(v)+u-v; & \frac12 \le v \le 1, 1-v \le u \le v, 
            \end{cases}$$
            and in particular $C=C^t$.
      \item[(b)] $$\displaystyle \int_0^1 \int_0^1 C(u,v) dudv = 4\int_0^{\frac12} \alpha_C(u)du - 4\int_{\frac12}^1 \alpha_C(u)du + 2\alpha_C(1)-\frac16.$$
      \item[(c)] The function $\delta_C$ is increasing and 1-Lipshitz on the interval $[0, \frac12]$.
      \item[(d)] For every $u \in [0, \frac12]$ we have $$\delta_C(1-u) = 1-2u+\delta_C(u),$$
      and in particular, $C=\widehat{C}$.
      \item[(e)] For every $u \in [0, \frac12]$ we have $$\alpha_C(1-u) = 2 \alpha_C({\textstyle\frac12}) - \alpha_C(u)+ \frac{(1-2u)^2}{4}.$$
      \item[(f)] $\alpha_C(1) = \frac{p+2}{6}$ and $\alpha_C(\frac12) = \frac{2p+1}{24}$.
      \item[(g)] $$\displaystyle \int_0^1 \int_0^1 C(u,v) dudv = 8\int_0^{\frac12} \alpha_C(u)du + \frac16.$$
      \item[(h)] For every $u \in [0, \frac12]$ we have $$\alpha_C(u) \ge \alpha_0(u) := \begin{cases}
            0; & 0 \le u \le u_0, \\
            \frac12(u-u_0)^2; & u_0 < u \le \frac12.
            \end{cases}$$
      \item[(i)] The following lower bound for $\rho(C)$ holds $$\rho(C) \ge \frac29\sqrt{3}(1+2p)^{3/2} -1.$$
  \end{enumerate}
\end{proposition}

\begin{proof}
From the condition $P(U=V) + P(U=1-V) =1$ it follows
\begin{align*}
  &P \left( U \leq \frac12, U < V < 1-U \right) = 0,  
  &P \left( U \geq \frac12, 1-U < V < U \right) = 0, \\
  &P \left( Y \leq \frac12, V < U < 1-V \right) = 0, 
  &P \left( Y \geq \frac12, 1-V < U < V \right) = 0.
\end{align*}
If $u \in [0, \frac12]$ we have for any $v \in  [u, 1-u]$ that $C(u,v) = C(u,u) = \delta_C(u)$. If $u \in [\frac12, 1]$ we have for any $v \in [1-u, u]$ that $C(u,v) = C(u,u) + v-u= \delta_C(u)+v-u$. Similar equalities hold if we interchange the roles of $u$ and $v$, so the copula $C$ is symmetric, which proves (a). 

To prove (b) we compute
\begin{align*}
\int_0^1 \int_0^1 C(u,v) dudv &= 2\int_0^{\frac12} \left(\int_u^{1-u} C(u,v) dv\right)du + 2 \int_{\frac12}^1 \left(\int_{1-u}^u C(u,v) dv\right)du \\
&= 2\int_0^{\frac12} \left(\int_u^{1-u} \delta_C(u) dv\right)du + 2 \int_{\frac12}^1 \left(\int_{1-u}^u (\delta_C(u)+v-u) dv\right)du \\
&= 2\int_0^{\frac12} \delta_C(u)(1-2u) du + 2 \int_{\frac12}^1  \delta_C(u)(2u-1) du  - \frac16 \\
&= 2\alpha_C({\textstyle\frac12}) - 4\int_0^{\frac12} u\delta_C(u)du -  2\alpha_C(1) + 2\alpha_C({\textstyle\frac12}) + 4 \int_{\frac12}^1  u\delta_C(u) du  - \frac16.
\end{align*}
Using integration by parts we obtain 
$$ \int_0^{\frac12} u\delta_C(u)du = {\textstyle \frac12\alpha_C(\frac12)} - \int_0^{\frac12} \alpha_C(u)du $$
and 
$$ \int_{\frac12}^1 u\delta_C(u)du = \alpha_C(1) - {\textstyle \frac12\alpha_C(\frac12)} - \int_{\frac12}^1 \alpha_C(u)du .$$
Plugging into the previous equation we obtain (b).

The diagonal of a copula is obviously increasing, so to prove (c) suppose $0 \le u \le v \le \frac12$. Since $C(u,v) = \delta_C(u)$ in this case, we have $\delta_C(v) - \delta_C(u) = C(v,v) - C(u,v) \le v-u$. If $u \in [0, \frac12]$ we obtain $\delta_C(1-u) = 1-2u+\delta_C(u)$ from $C(1-u, 1-u) = C(1-u, u) + 1-2u = C(u,u) +1-2u$. If $u \in [0, \frac12]$ and $v \in [u, 1-u]$, we have 
\begin{align*}
    \widehat{C}(u,v) &= u+v-1+C(1-u,1-v) = u+v-1+\delta_C(1-u)+(1-v)-(1-u)\\
    &= \delta_C(1-u) - 1 + 2u = \delta_C(u) = C(u,v).
\end{align*}
Similarly other cases are treated to prove (d). 

To prove (e) assume $u \in [0, \frac12]$ and compute
\begin{align*}
\alpha_C(1-u) &=  \int_0^{1-u} \delta_C(t) dt = \alpha_C({\textstyle\frac12}) + \int_{\frac12}^{1-u} \delta_C(t) dt \
= \alpha_C({\textstyle\frac12}) + \int_{\frac12}^{1-u} (2t-1+\delta_C(1-t)) dt \\
&= \alpha_C({\textstyle\frac12}) + \frac{(1-2u)^2}{4} + \int_{\frac12}^{1-u} \delta_C(1-t) dt 
= \alpha_C({\textstyle\frac12}) + \frac{(1-2u)^2}{4} - \int_{\frac12}^u \delta_C(t) dt \\
&= 2\alpha_C({\textstyle\frac12}) + \frac{(1-2u)^2}{4} - \int_0^u \delta_C(t) dt
= 2\alpha_C({\textstyle\frac12}) + \frac{(1-2u)^2}{4} - \alpha_C(u).
\end{align*}

The property $\alpha_C(1) = \frac{p+2}{6}$ is immediate from the definitions of $\phi(C)$ and $\alpha_C$, and we get $\alpha_C(\frac12) = \frac{2p+1}{24}$ by plugging $u=0$ into (e).

To prove (g) we apply (e) in (b):
\begin{align*}
\int_0^1 \int_0^1 C(u,v) dudv &= 4\int_0^{\frac12} \alpha_C(u)du - 4\int_{\frac12}^1 \alpha_C(u)du + 2\alpha_C(1)-\frac16\\
& = 4\int_0^{\frac12} \alpha_C(u)du - 4\int_{\frac12}^1 \left(2 \alpha_C({\textstyle\frac12}) - \alpha_C(1-u)+ \frac{(1-2u)^2}{4}\right)du + 2\alpha_C(1)-\frac16 \\
& = 4\int_0^{\frac12} \alpha_C(u)du + 4\int_{\frac12}^1 \alpha_C(1-u)du - 4\alpha_C({\textstyle\frac12}) + 2\alpha_C(1)-\frac13 \\
& = 8\int_0^{\frac12} \alpha_C(u)du - 4\alpha_C({\textstyle\frac12}) + 2\left(2\alpha_C({\textstyle\frac12})+\frac14\right)-\frac13 \\
& = 8\int_0^{\frac12} \alpha_C(u)du + \frac16.
\end{align*}

We will prove (h) by contradiction. First notice that $\alpha_C(u) \ge 0$ and 
$\alpha_C(\frac12) = \alpha_0(\frac12) = \frac{2p+1}{24}$ by the definition of $u_0$. So suppose that $\alpha_C(u_1) < \alpha_0(u_1) = \frac12 (u_1-u_0)^2$ for some $u_1 \in (u_0, \frac12)$. We will first prove that $\delta_1 := \delta_C(u_1) < u_1-u_0$. Since $\delta_C$ is 1-Lipshitz, we have for every $t \in [u_0, u_1]$ that $\delta_C(t) \ge \delta_1 + t - u_1$. Thus
$$\alpha_C(u_1) = \int_0^{u_1} \delta_C(t)dt \ge \int_{u_0}^{u_1} \delta_C(t)dt \ge \int_{u_0}^{u_1} (\delta_1 + t - u_1)dt  = \delta_1(u_1-u_0) - \frac12 (u_1-u_0)^2$$
and
$$\delta_1\le \frac{\alpha_C(u_1) + \frac12 (u_1-u_0)^2}{u_1-u_0} < \frac{\frac12 (u_1-u_0)^2 + \frac12 (u_1-u_0)^2}{u_1-u_0} = u_1 - u_0.$$
Now we have for every $t \in [u_1, \frac12]$ that $\delta_C(t) \le \delta_1 + t - u_1 < t - u_0$. Hence
\begin{align*}
 \alpha_C({\textstyle\frac12}) &= \alpha_C(u_1) + \int_{u_1}^{\frac12} \delta_C(t)dt \\
 &< \alpha_C(u_1) + \int_{u_1}^{\frac12} (t - u_0)dt \\  
 &= \textstyle \alpha_C(u_1) + \frac18 - \frac12u_1^2 - \frac12 u_0(1-2u_1) \\
 &< \textstyle \frac12 (u_1-u_0)^2 + \frac18 - \frac12u_1^2 - \frac12 u_0(1-2u_1) \\
 &= \textstyle \frac12 (\frac12-u_0)^2 = \alpha_0(\frac12), 
\end{align*}
a contradiction. 

Finally, using the definition of $\rho$, (g), and (h), we have
\begin{align*}
\rho(C) &= 12 \int_0^1 \int_0^1 C(u,v) dudv -3 = 96 \int_0^{\frac12} \alpha_C(u)du -1   \\
&\ge 96 \int_0^{\frac12} \alpha_0(u)du -1 = 48 \int_{u_0}^{\frac12} (u-u_0)^2du -1 \\
& = 2(1-2u_0)^3 - 1 = \frac29\sqrt{3}(1+2p)^{3/2} -1.
\end{align*}
\end{proof}

In the following example, we note that all points on the curve $r=\frac29\sqrt{3}(1+2p)^{3/2} -1$ can be attained by shuffles of $M$.

\begin{example} \label{ex:1}
Let $a \in [0, \frac12]$ and let $C_a$ be a shuffle of $M$
$$ C_a = M(3, (a, 1-a), (3, 2, 1), (-1, 1, -1)).$$
Notice that $C_0 = M$ and $C_{\frac12} = W$. We have
$$ \delta_{C_a}(u) = \begin{cases}
			0; & u \le a, \\
			u-a; & a \le u \le 1-a, \\
			2u-1; & 1-a \le u \le 1,
    \end{cases} $$
and
$$ C_a(u, v) = \begin{cases}
            0; & 0 \le u \le a, u \le v \le 1-u, \\
            0; & 0 \le v \le a, v \le u \le 1-v, \\
            u+v-1; & 1-a \le u \le 1, 1-u \le v \le u, \\
            u+v-1; & 1-a \le v \le 1, 1-v \le u \le v, \\
            u-a; & a \le u \le 1-a, u \le v \le 1-a, \\
            v-a; & a \le u \le 1-a, a \le v \le u, 
            \end{cases}$$
It follows that
$$ \phi(C_a) = 6a^2 - 6a + 1 =\tfrac{3}{2}(1-2a)^2-\tfrac{1}{2} $$ 
and  
$$ \rho(C_a) = -16a^3 + 24a^2 - 12a + 1 = 2(1-2a)^3-1,$$
so that $\rho(C_a) = \frac29\sqrt{3}(1+2\phi(C_a))^{3/2} -1$, the point $(\phi(C_a), \rho(C_a))$ lies on the curve $r=\frac29\sqrt{3}(1+2p)^{3/2} -1$ and every point on this curve for $p \in [-\frac12, 1]$ is attained.  The scatterplot of copula $C_a$ is shown in Figure \ref{fig:ex}.  \end{example}

\begin{figure}[ht] 
    \centering
    \includegraphics[width=5cm]{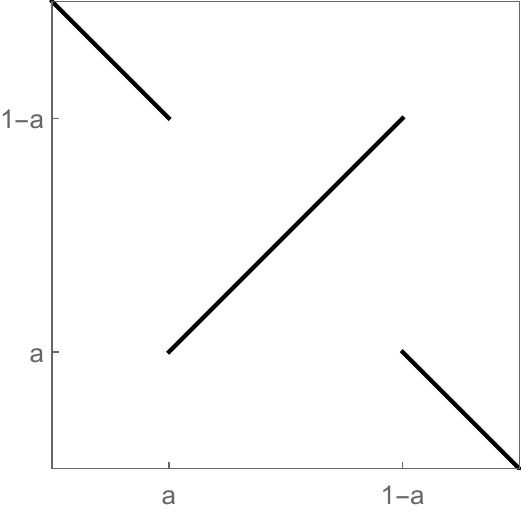}
    \caption{The scatterplot of copula $C_a$ from Example \ref{ex:1}. }
    \label{fig:ex}
\end{figure}

\section{Lower bound - general case}\label{sec:general}

Let $U$ and $V$ be uniformly distributed random variables on interval $\II$ and let $C\in\CC$ be their copula.
In this section we are going to prove the lower bound
$$\frac29\sqrt{3}(1+2\phi(C))^{3/2}-1 \le \rho(C)$$
holds for arbitrary copula $C$.
Define two function depending on copula $C \in \CC$ by
$$f(C) = \rho(C) - \frac29\sqrt{3}(1+2\phi(C))^{3/2} +1$$
and
$$q(C) = P(U=V) + P(U=1-V) .$$
Note that in order to prove the above bound we need to prove that $f(C) \geq 0$ for any copula $C$.
We will reduce the general case to the special case considered in Section~\ref{sec:special} by redistribution of the mass of the copula.
In each step the value of $f(C)$ will decrease while the value of $q(C)$ will increase, until $q(C)$ becomes $1$. 
We believe that this method of mass shifting may also be useful in other considerations. Next lemma describes a single step.

\begin{lemma}\label{lem:step}
Let $C \in \CC$ be a doubly symmetric shuffle of $M$, $C = M(n, J, \pi, \omega)$. Write $J=\{J_1, ..., J_n\}$. Suppose that for some $i\in \{1, 2, ..., n\}$ we have $J_i \times J_{\pi(i)} = [a, a+x] \times [b, b+x]$ with $a < b < 1-a$, $\omega(i) = 1$, and $x > 0$. Then there exists doubly symmetric shuffle of $M$, $C' = M(n, J, \pi', \omega')$ such that 
$f(C') < f(C)$ and $q(C') = q(C) + 2x$ in the case $\pi(i) =n+1-i$ or $q(C') = q(C) + 4x$ otherwise.
\end{lemma}

\begin{proof}
We have $$V_C(J_i \times J_{\pi(i)}) = V_C([a, a+x] \times [b, b+x])=x$$ and corresponding reflected squares are $J_{\pi(i)} \times J_i = [b, b+x] \times [a, a+x]$, $J_{n+1-i} \times J_{\pi(n+1-i)} = [1-a-x, 1-a] \times [1-b-x, 1-b]$, and $J_{\pi(n+1-i)} \times J_{n+1-i}=[1-b-x, 1-b] \times [1-a-x, 1-x]$.
Furthermore, $\omega(\pi(i)) = \omega(n+1-i) = \omega(\pi(n+1-i)) = 1$. Since $a<b<1-a$, we have $a<\frac12$, $i \le \frac{n}{2}$ and $i < \pi(i) \le n+1-i$.
We will consider several cases.

{\it Case I:} Suppose that $\pi(i) = n+1-i$, so $b=1-a-x$ and in $C$ we have two segments crossing the opposite diagonal. Define
$$\pi'=\pi \text{ and } \omega'(j) = \begin{cases}
			-1; & j \in \{i,n+1-i\}, \\
			\omega(j); &\text{otherwise}.
    \end{cases} $$
The copula $C'$ is doubly symmetric shuffle. Figure \ref{fig:C-C'-I} shows
the mass distribution of the difference of copulas $C$ and $C'$, and the graph of the function $C-C'$. The mass of $C'$ is negative in the difference, it is shown dashed.

\begin{figure}[!ht]
    \centering
    \includegraphics{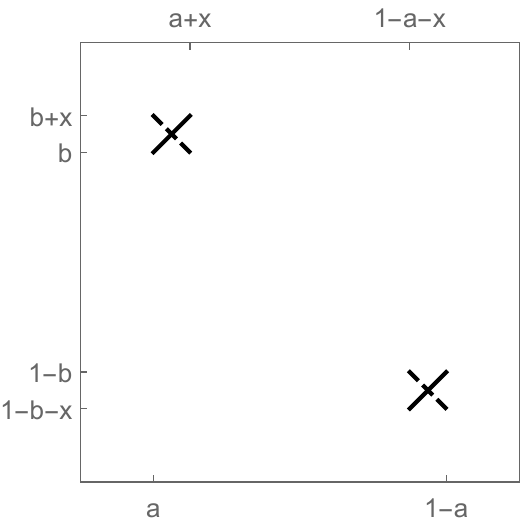}
    \includegraphics{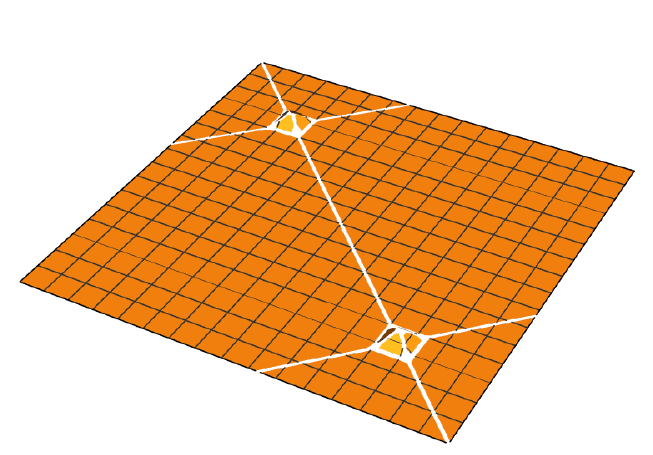}
    \caption{The mass distribution (left) and the graph (right) of the difference of copulas considered in the proof of Lemma~\ref{lem:step} in \emph{Case I}.}
    \label{fig:C-C'-I}
\end{figure}

We have 
\begin{align*}
C(u,v) - C'(u,v) &= \max\{0, \min\{u-a, v-(1-a-x), 1-a-v, a+x-u\}\} \\
                 & \ \ \  + \max\{0, \min\{v-a, u-(1-a-x), 1-a-u, a+x-v\}\},
\end{align*}
so $\delta_C = \delta_{C'}$ and $\phi(C) = \phi(C')$. Furthermore,
$$\rho(C) - \rho(C') = 12 \int_0^1 \int_0^1 (C(u,v) - C'(u,v))dudv = 24V,$$
where $V$ is the volume of a pyramid having the base a square with the side $x$  and height $\frac{x}{2}$, so $V=\frac{x^3}{6}$. We thus have $f(C) = f(C')+4x^3 > f(C')$, and since the two new segments lie on the opposite diagonal also
$q(C') = q(C) + 2x$. 

{\it Case II:} Suppose that $\frac{n}{2} < \pi(i) < n+1-i$, so $\frac12 \le b \le 1-a-2x$. Define
\begin{align*}
\pi'(j)&=\begin{cases}
			n+1-j; & j \in \{i,n+1-i,\pi(i),\pi(n+1-i)\}, \\
			\pi(j); &\text{otherwise}, 
		\end{cases}
			\text{ and } \\
\omega'(j) &= \begin{cases}
			-1; & j \in \{i,n+1-i,\pi(i),\pi(n+1-i)\}, \\
			\omega(j); &\text{otherwise}.
    \end{cases} 
\end{align*}
Figure \ref{fig:C-C'-II} shows
the mass distribution of the difference of copulas $C$ and $C'$, and the graph of the function $C-C'$.

\begin{figure}[!ht]
    \centering
    \includegraphics{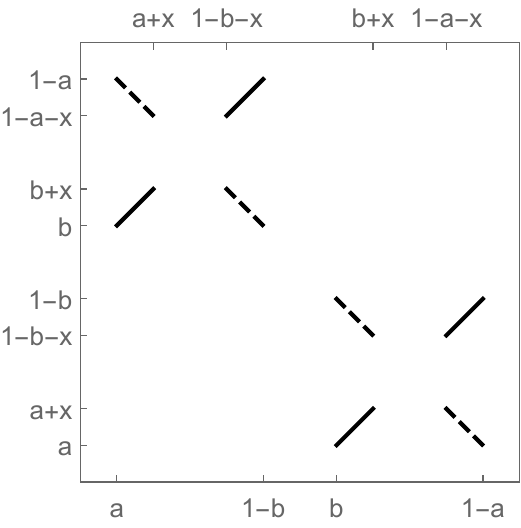}
    \includegraphics{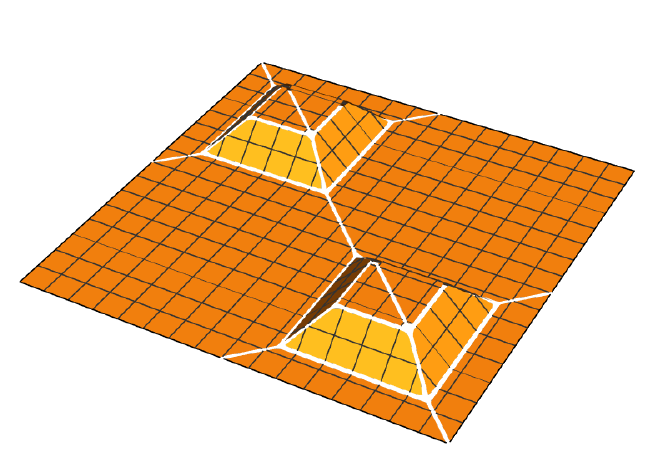}
    \caption{The mass distribution (left) and the graph (right) of the difference of copulas considered in the proof of Lemma~\ref{lem:step} in \emph{Case II}.}
    \label{fig:C-C'-II}
\end{figure}

We have 
\begin{align*}
C(u,v) - C'(u,v) &= \max\{0, \min\{u-a, v-b, 1-a-v, 1-b-u, x\}\} \\
                 & \ \ \ + \max\{0, \min\{v-a, u-b, 1-a-u, 1-b-v, x\}\},
\end{align*}
so again $\delta_C = \delta_{C'}$ and $\phi(C) = \phi(C')$. Furthermore,
$\rho(C) - \rho(C') = 24V,$
where $V$ is the volume of a square frustum having the lower base a square with the side $1-a-b$, the upper base a square with the side $1-a-b-2x$ and height $x$, so 
$$V=\textstyle \frac16 (1-a-b)^3 - \frac16 (1-a-b-2x)^3 = x(1-a-b-x)^2 + \frac13 x^3.$$ 
We thus have $f(C) = f(C')+24x(1-a-b-x)^2 + 8x^3 > f(C')$, and since the four new segments lie on the opposite diagonal also
$q(C') = q(C) + 4x$.

{\it Case III:} Suppose that $\pi(i) \le \frac{n}{2}$, so $b \le \frac12-x$ and assume also $b \ge a + \frac{1}{\sqrt{3}}\sqrt{1+2\phi(C)}$. Define as in Case II
\begin{align*}
\pi'(j)&=\begin{cases}
			n+1-j; & j \in \{i,n+1-i,\pi(i),\pi(n+1-i)\}), \\
		    \pi(j); &\text{otherwise}, 
		\end{cases}
			\text{ and } \\
\omega'(j) &= \begin{cases}
			-1; & j \in \{i,n+1-i,\pi(i),\pi(n+1-i)\}, \\
			\omega(j); &\text{otherwise}.
    \end{cases} 
\end{align*}
Figure \ref{fig:C-C'-III} shows
the mass distribution and the graph of the function $C-C'$ in this case.

\begin{figure}[!ht]
    \centering
    \includegraphics{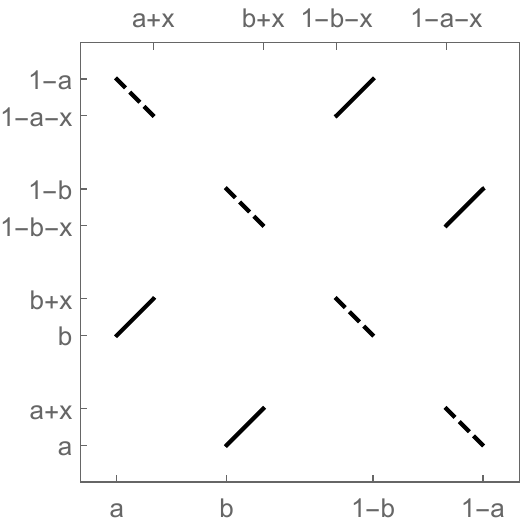}
    \includegraphics{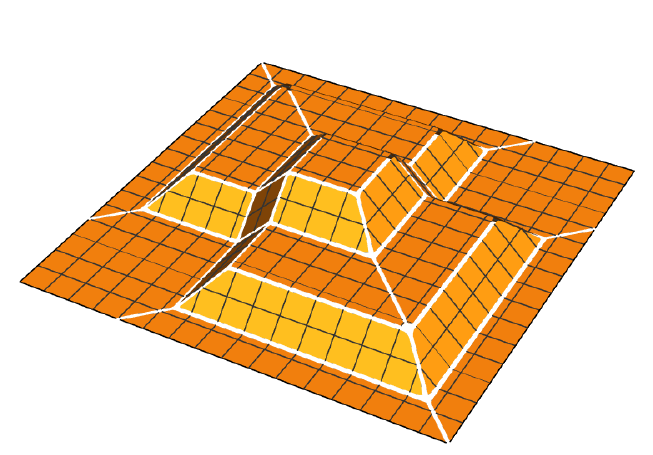}
    \caption{The mass distribution (left) and the graph (right) of the difference of copulas considered in the proof of Lemma~\ref{lem:step} in \emph{Case III}.}
    \label{fig:C-C'-III}
\end{figure}

Again we have 
\begin{align*}
C(u,v) - C'(u,v) &= \max\{0, \min\{u-a, v-b, 1-a-v, 1-b-u, x\}\} \\
                 & \ \ \ + \max\{0, \min\{v-a, u-b, 1-a-u, 1-b-v, x\}\},
\end{align*}
so again $\rho(C) = \rho(C')+24x(1-a-b-x)^2 + 8x^3$. But now
$$\delta_C(u) - \delta_{C'}(u) = 2\max\{0, \min\{u-b, 1-b-u, x\}\},$$
so 
$$\phi(C) = \phi(C') + 6 \int_0^1 (\delta_C(u) - \delta_{C'}(u))du = 3(1-2b)^2-3(1-2b-2x)^2 = 12x(1-2b-x).$$ 
Denote by $d=12x(1-2b-x)$. Now
\begin{align*}
f(C)-f(C') &= \rho(C) - \rho(C') - \frac29\sqrt{3}(1+2\phi(C))^{3/2} +\frac29\sqrt{3}(1+2\phi(C'))^{3/2} \\
& = 24x(1-a-b-x)^2 + 8x^3 + \frac29\sqrt{3}\left((1+2\phi(C)-2d)^{3/2} - (1+2\phi(C))^{3/2}\right).
\end{align*}
Using Lagrange theorem there exists $t\in [0, d]$ such that
$$(1+2\phi(C)-2d)^{3/2} - (1+2\phi(C))^{3/2} = - 3d (1+2\phi(C)-2t)^{1/2} \ge - 3d (1+2\phi(C))^{1/2},$$
so
\begin{align*}
f(C)-f(C') &\ge 24x(1-a-b-x)^2 + 8x^3 - \frac29\sqrt{3}\cdot 3d(1+2\phi(C))^{1/2} \\
& = 24x(1-a-b-x)^2 + 8x^3 - \frac{2d}{\sqrt{3}}\sqrt{1+2\phi(C)} \\
& = 24x(1-a-b-x)^2 + 8x^3 - 24x(1-2b-x)\frac{1}{\sqrt{3}}\sqrt{1+2\phi(C)} \\
& \ge 24x(1-a-b-x)^2 + 8x^3 - 24x(1-a-b-x)\frac{1}{\sqrt{3}}\sqrt{1+2\phi(C)} \\
& = 8x^3 + 24x(1-a-b-x)\left(1-a-b-x - \frac{1}{\sqrt{3}}\sqrt{1+2\phi(C)}\right)
\end{align*}
We now use the assumptions $b \le \frac12-x$, so $1 \ge 2b+2x$, and $b \ge a+\frac{1}{\sqrt{3}}\sqrt{1+2\phi(C)}$ to estimate further
\begin{align*}
f(C)-f(C') & \ge 8x^3 + 24x(1-a-b-x)\left(2b+2x-a-b-x - \frac{1}{\sqrt{3}}\sqrt{1+2\phi(C)}\right) \\
& = 8x^3 + 24x(1-a-b-x)\left(x + b - a - \frac{1}{\sqrt{3}}\sqrt{1+2\phi(C)}\right) \\
& \ge 8x^3 + 24x^2(1-a-b-x) >0. 
\end{align*}
Finally,  since the four new segments lie on the opposite diagonal,  we have 
$q(C') = q(C) + 4x$ as in the previous case. 

{\it Case IV:} Suppose that $i< \pi(i) \le \frac{n}{2}$, so $a+x \le b \le \frac12-x$ and assume also $b < a + \frac{1}{\sqrt{3}}\sqrt{1+2\phi(C)}$. Define $$\pi'(j)=\begin{cases}
			j; & j\in \{i,n+1-i,\pi(i),\pi(n+1-i)\}, \\
			\pi(j); &\text{otherwise}, 
		\end{cases}
			\text{ and } \\
\omega' = \omega.$$
Figure \ref{fig:C-C'-VI} shows
the mass distribution and the graph of the function $C-C'$.

\begin{figure}[!ht]
    \centering
    \includegraphics{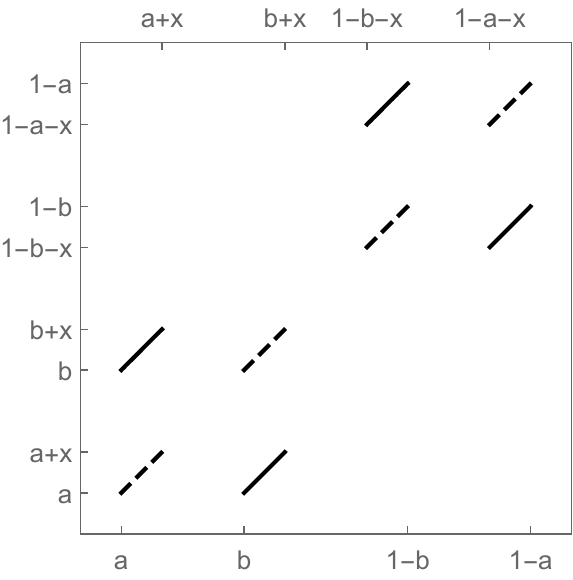}
    \includegraphics{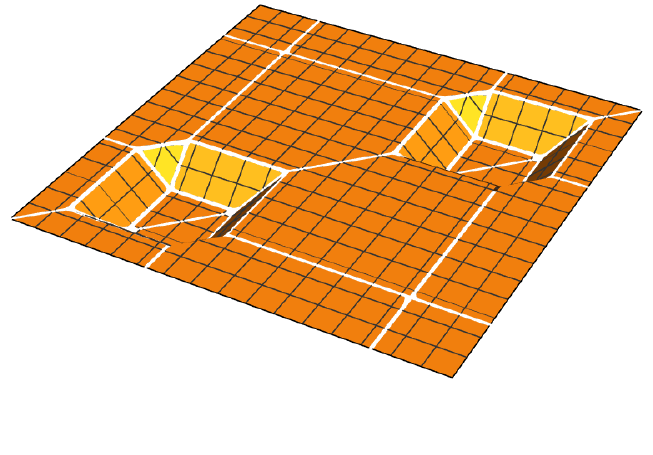}
    \caption{The mass distribution (left) and the graph (right) of the difference of copulas considered in the proof of Lemma~\ref{lem:step} in \emph{Case VI}.}
    \label{fig:C-C'-VI}
\end{figure}

This time we have $C' \ge C$ so
\begin{align*}
C(u,v) - C'(u,v) &= -(C'(u,v) - C(u,v)) \\
&= -\max\{0, \min\{u-a, v-a, b+x-u, b+x-v, x, \\
                 & \ \ \ \ \ \ \ \ \ \ \ \ \ \ \ \ \ \ \ \ b-a+u-v, b-a+v-u\}\} \\
                 & \ \ \ - \max\{0, \min\{u-1+b+x, v-1+b+x, 1-a-u, 1-a-v, x, \\     
                 & \ \ \ \ \ \ \ \ \ \ \ \ \ \ \ \ \ \ \ \ b-a+u-v, b-a+v-u\}\}.
\end{align*}
It follows that
$$\delta_{C}(u)  - \delta_{C'}(u) = -\max\{0, \min\{u-a, b+x-u, x\}\}-\max\{0, \min\{u-1+b+x, 1-a-u, x\}\},$$
so $\phi(C) = \phi(C') - 12x(b-a)$. Furthermore $\rho(C) = \rho(C') - 24V,$
where $V$ is the volume of a square frustum having the lower base a square with the side $b-a+x$, the upper base a square with the side $b-a-x$, height $x$, and two corners cut off, so
$$V=\textstyle \frac16 (b-a+x)^3 - \frac16 (b-a-x)^3 - 2 \cdot \frac16 x^3 = x(b-a)^2$$  
and $\rho(C) = \rho(C') - 24x(b-a)^2.$
Similarly as in the previous case we estimate
\begin{align*}
f(C)-f(C') &= \rho(C) - \rho(C') - \frac29\sqrt{3}(1+2\phi(C))^{3/2} +\frac29\sqrt{3}(1+2\phi(C'))^{3/2} \\
& = - 24x(b-a)^2 + \frac29\sqrt{3}\left((1+2\phi(C)+2d)^{3/2} - (1+2\phi(C))^{3/2}\right)
\end{align*}
where $d= 12x(b-a)$, so there exists $t\in [0, d]$ such that
\begin{align*}
f(C)-f(C') &= - 24x(b-a)^2 + \frac29\sqrt{3}\cdot 3d(1+2\phi(C)+2t)^{1/2} \\
& \ge - 24x(b-a)^2 + \frac{2d}{\sqrt{3}}\sqrt{1+2\phi(C)} \\
& = - 24x(b-a)^2 + 24x(b-a)\frac{1}{\sqrt{3}}\sqrt{1+2\phi(C)} \\
& = 24x(b-a)\left(a-b+\frac{1}{\sqrt{3}}\sqrt{1+2\phi(C)} \right) > 0.
\end{align*}
Finally,  since the four new segments lie on the main diagonal,  we have 
$q(C') = q(C) + 4x$ as in the previous case. 
\end{proof}

We are now finally ready to prove our the lower bound.

\begin{theorem}\label{thm:main}
For any copula $C$ we have
$$ \frac29\sqrt{3}(1+2\phi(C))^{3/2}-1 \le \rho(C).$$
For any value $\phi(C) \in [-\frac12,1]$ the bound is attained by some shuffle of $M$.
\end{theorem}

\begin{proof}
Let $C$ be a copula such that $\phi(C)=p$ for some $p\in[-\frac12,1]$. Define a copula $\widetilde{C}=\frac14(C+C^t+\widehat{C}+\widehat{C}^t)$.
Since the transformations $C \mapsto C^t$ and $C \mapsto \widehat{C}$ commute it is easy to verify that $\widetilde{C}$ is a doubly symmetric copula.
Furthermore, $\phi(\widetilde{C})=\phi(C)$ and $\rho(\widetilde{C})=\rho(C)$ because $\rho$ and $\phi$ preserve convex combinations and $\rho(C^t)=\rho(\widehat{C})=\rho(C)$ and $\phi(C^t)=\phi(\widehat{C})=\phi(C)$.
Thus, by replacing $C$ with $\widetilde{C}$, we may assume without loss of generality that $C$ is a doubly symmetric copula.

Let $\varepsilon >0$. By Lemma~\ref{lem:approx} there exists a straight doubly symmetric shuffle $C'=M(n,J,\pi,\omega)$ such that
$\sup_{(u,v) \in \II^2} |C(u,v)-C'(u,v)| <\varepsilon.$
Hence,
\begin{equation}\label{eq:est}
|\rho(C)-\rho(C')| < 12\varepsilon \qquad\mbox{and}\qquad |\phi(C)-\phi(C')| < 6\varepsilon.
\end{equation}
By Lagrange theorem we have
\begin{equation}\label{eq:f}
\begin{aligned}
    f(C)-f(C') &=\rho(C)-\rho(C')-\Big(\frac29\sqrt{3}(1+2\phi(C))^{3/2}-\frac29\sqrt{3}(1+2\phi(C'))^{3/2}\Big)\\
    &= \rho(C)-\rho(C')-\frac{2}{\sqrt{3}}(1+2t)^{1/2}\big(\phi(C)-\phi(C')\big)
    \end{aligned}
\end{equation}
for some $t$ between $\phi(C)$ and $\phi(C')$, so that $t\le 1$. Using the estimates \eqref{eq:est} in equation \eqref{eq:f} and the estimate for $t$ we get
\begin{equation}\label{eq:f2}
    f(C)-f(C')>-12\varepsilon-\frac{2}{\sqrt{3}}(1+2t)^{1/2} \cdot 6\varepsilon \ge -24\varepsilon.
\end{equation}
By Lemma~\ref{lem:shuffle_prop} all the mass of $C'$ is concentrated on squares $J_i \times J_{\pi(i)}$, $i=1,2,\ldots,n$, and the squares $J_i \times J_{\pi(i)}$, $ J_{\pi(i)} \times J_{i}$, $J_{n-i+1} \times J_{\pi(n-i+1)}$ and $J_{\pi(n-i+1)} \times J_{n-i+1}$ have the same $C'$-volume.
So, as long as $q(C')<1$, there exists $i\in \{1,2,\ldots,n\}$ such that $J_i \times J_{\pi(i)} = [a, a+x] \times [b, b+x]$ with $a < b < 1-a$ and $x>0$.
We now apply Lemma~\ref{lem:step} to all squares with this property, one at a time. Each time we apply the lemma, the function $q(C')$ increases and the function $f(C')$ decreases.
When this process ends, we are left with a doubly symmetric shuffle $C''$ such that $q(C'')=1$ and $f(C'') \leq f(C')$.
By Proposition~\ref{prop:x} $(i)$, we have $f(C'') \ge 0$, and hence $f(C') \ge 0$.
Together with the estimate \eqref{eq:f2}, this implies
$f(C) > -24\varepsilon$.
Sending $\varepsilon$ to $0$ we obtain $f(C) \geq 0$, which proves the bound.

By Example~\ref{ex:1} the bound is attained by a shuffle of $M$.
\end{proof}

\section{Upper bound}\label{sec:upper}

We will first prove the estimate for the upper bound of $\rho(C)$ in terms of $\phi(C)$ in the case when $C$ is the diagonal copula $K_\delta$ (see below) and the diagonal is nice enough.

Let $C$ be a doubly symmetric copula. Then its diagonal $\delta(u) = C(u, u)$ satisfies 
\begin{equation}\label{eq:delta-sym}
    \delta(u) = 2u-1 +\delta(1-u).
\end{equation}
We call such a diagonal \emph{symmetric diagonal}.
It is well known that the diagonal is increasing and 2-Lipshitz, so it is differentiable almost everywhere on $\II$ and $\delta'(u) \in [0,2]$ where it exists. We are going to assume that it is differentiable everywhere, its derivative is continuous, and that
\begin{equation}\label{eq:delta-smooth}
    0 < \delta'(u) < 2 \text{ for all except possibly finitely many points } u \in \II.
\end{equation}
If follows that $\delta$ is strictly increasing, so it is bijective and its inverse $\delta^{-1}: \II \to \II$ exists. 
Let us introduce three auxiliary functions $\alpha, g, h: \II \to \RR$ as follows
\begin{align}\label{eq:aux}
    \alpha(u) &= \int_0^u \delta(t)dt, \nonumber\\
    g(u) &= 2u - \delta(u), \\
    h(u) &= \delta^{-1}(g(u)). \nonumber
\end{align}
It holds that $g(0) = h(0) = 0, g(1) = h(1) = 1$. Since $\delta(u) \le u$ for $u \in \II$, we have $g(u) \ge u$ and $h(u) \ge u$ for every $u \in \II$. It follows immediately from  \eqref{eq:delta-smooth} that $g$ and $h$ are bijective so that the inverses $g^{-1}, h^{-1}: \II \to \II$ exist. Furthermore, $g$ is differentiable everywhere on $\II$ and its derivative is continuous. Also, by \eqref{eq:delta-smooth}, $h$ is differentiable everywhere except possibly in finitely many points $u \in \II$ and its derivative is continuous where it exists.

Fredrichs and Neslen in \cite{FrNe97} introduced \emph{diagonal copula} $K_\delta$, given by
$$ K_\delta(u,v) = \min\left\{u,v, \frac{\delta(u)+\delta(v)}{2}\right\}.$$
Notice that $u = \frac12(\delta(u)+\delta(v))$ holds if and only if $v = h(u)$, so that
\begin{equation}\label{eq:K_delta}
    K_\delta(u,v) =  \begin{cases}
            u; & v \ge h(u), \\
            \frac{\delta(u)+\delta(v)}{2}; & v < h(u), u < h(v), \\
            v; & u \ge h(v). 
            \end{cases}
\end{equation}

The following proposition establishes the upper bound for the diagonal copula by simplifying its double integral over the unit square $\II^2$.

\begin{proposition}\label{prop:y}
  Let $\delta$ be a symmetric differentiable diagonal with continuous derivative that satisfies \eqref{eq:delta-smooth}, $\alpha, g, h:\II \to \II$ auxiliary functions defined by \eqref{eq:aux}, and $K_\delta$ diagonal copula. Then the following holds: 
  \begin{enumerate}
      \item[(a)] $$g^{-1}(u) = 1 - \delta^{-1}(1-u) \text{ for all } u \in \II.$$
      \item[(b)] $$\int_0^1 u\delta(u)du = \frac12\alpha(1)+\frac{1}{12}.$$
      \item[(c)] $$\int_0^1 \alpha(u)du = \frac12\alpha(1)-\frac{1}{12}.$$
      \item[(d)] $$\int_0^1 \delta^{-1}(u)du = 1 - \alpha(1).$$
      \item[(e)] $$\int_0^1 (\delta^{-1}(u))^2du = \frac56 - \alpha(1).$$
      \item[(f)] $$\int_0^1 \alpha(h(u))du = \alpha(1) - 1 + \int_0^1 (4u-\delta(u)-u\delta'(u))h(u)du.$$
      \item[(g)] $$\int_0^1 \int_0^1 K_\delta(u,v) dudv = \int_0^1 g^{-1}(u)\delta^{-1}(u)du.$$
      \item[(h)] $$\int_0^1 \int_0^1 K_\delta(u,v) dudv \le 2\alpha(1) - 2\alpha(1)^2 - \frac16.$$
      \item[(i)] $$\rho(K_\delta) \le 1-\frac23 (1-\phi(K_\delta))^2.$$
  \end{enumerate}
\end{proposition}

\begin{proof}
We first compute
\begin{align*}
   g(1 - \delta^{-1}(1-u)) &= 2(1 - \delta^{-1}(1-u)) - \delta(1 - \delta^{-1}(1-u)) \\
   &= 2 - 2\delta^{-1}(1-u) - \left( 1 - 2\delta^{-1}(1-u) + \delta(\delta^{-1}(1-u))\right) 
\end{align*}
by equality \eqref{eq:delta-sym}, so that
$$ g(1 - \delta^{-1}(1-u)) = 2 - 2\delta^{-1}(1-u) - 1 + 2\delta^{-1}(1-u) - (1-u) = u, $$
which proves (a). 

Next we use equality \eqref{eq:delta-sym} in the integral
$$ \int_0^1 u\delta(u)du = \int_0^1 u(2u-1+\delta(1-u))du = \frac16 + \int_0^1 u\delta(1-u)du.$$
We introduce a new variable $t = 1-u$ to get
\begin{align*}
\int_0^1 u\delta(u)du &= \frac16 + \int_1^0 (1-t)\delta(t)(-dt) = \frac16 + \int_0^1 \delta(t)dt  - \int_0^1 t\delta(t)dt \\
&= \frac16 + \alpha(1) - \int_0^1 u\delta(u)du.
\end{align*}
We express the integral from the obtained equation to get (b). 

To prove (c) we integrate by parts and use (b)
$$ \int_0^1 \alpha(u)du = u\alpha(u)\bigg |_0^1 - \int_0^1 u\delta(u)du =  \alpha(1) - \left(\frac12\alpha(1)+\frac{1}{12}\right) = \frac12\alpha(1)-\frac{1}{12}. $$

To prove (d) we first introduce a new variable by $u=\delta(t)$ in the integral and then integrate by parts
$$ \int_0^1 \delta^{-1}(u)du = \int_0^1 t\delta'(t)dt = t\delta(t)\bigg |_0^1 - \int_0^1 \delta(t)dt = 1 - \alpha(1). $$

In a similar way we prove (e), in the final step we use (b)
$$ \int_0^1 (\delta^{-1}(u))^2du = \int_0^1 t^2\delta'(t)dt = t^2\delta(t)\bigg |_0^1 - \int_0^1 2t\delta(t)dt = 1 - 2\left(\frac12\alpha(1)+\frac{1}{12}\right) = \frac56 - \alpha(1). $$

To prove (f)  we  integrate by parts twice
\begin{align*}
\int_0^1 \alpha(h(u))du &= u\alpha(h(u))\bigg |_0^1 - \int_0^1 u\alpha'(h(u))h'(u)du = \alpha(1) - \int_0^1 ug(u)h'(u)du \\
&= \alpha(1) - ug(u)h(u)\bigg |_0^1 + \int_0^1 (g(u) + ug'(u))h(u)du \\
&= \alpha(1) - 1 + \int_0^1 (4u-\delta(u)-u\delta'(u))h(u)du.
\end{align*}

To prove (g) we first use the symmetry of copula $K_\delta$ and then equation \eqref{eq:K_delta}
\begin{align*}
I &= \int_0^1 \int_0^1 K_\delta(u,v) dudv = 2\int_0^1 \left( \int_u^1 K_\delta(u,v) dv\right) du \\
&= 2\int_0^1 \left( \int_u^{h(u)} \frac{\delta(u)+\delta(v)}{2} dv + \int_{h(u)}^1 u dv\right) du \\ 
&= \int_0^1 \left(\delta(u)(h(u)-u)+\alpha(h(u)) -\alpha(u)\right) du + 2\int_0^1 u(1-h(u)) du \\ 
&= \int_0^1 \delta(u)h(u)du - \int_0^1 u\delta(u)du + \int_0^1 \alpha(h(u)) du  - \int_0^1 \alpha(u)du - \int_0^1 2uh(u) du  +1. 
\end{align*}
Next we use (b), (c), and (f) to get
\begin{align*}
I &= \int_0^1 \delta(u)h(u)du - \left(\frac12\alpha(1)+\frac{1}{12}\right) + \alpha(1) - 1 + \int_0^1 (4u-\delta(u)-u\delta'(u))h(u)du  \\
& \ \ \ \ \ - \left(\frac12\alpha(1)-\frac{1}{12}\right) - \int_0^1 2uh(u) du  +1 \\
&= \int_0^1 (2u - u\delta'(u))h(u)du = \int_0^1 ug'(u)h(u)du. 
\end{align*}
Now we introduce a new variable $t=g(u)$ to get
$$ I = \int_0^1 g^{-1}(t)h(g^{-1}(t))dt = \int_0^1 g^{-1}(t)\delta^{-1}(t)dt.$$

To prove (h) we first use (g), (a), and (d)
\begin{align*}
I &= \int_0^1 \int_0^1 K_\delta(u,v) dudv = \int_0^1 g^{-1}(u)\delta^{-1}(u)du \\
&= \int_0^1 (1 - \delta^{-1}(1-u))\delta^{-1}(u)du = 1 - \alpha(1) - \int_0^1 \delta^{-1}(1-u)\delta^{-1}(u)du.
\end{align*}
Now
$$ \delta^{-1}(1-u)\delta^{-1}(u) = \frac12 \left((\delta^{-1}(1-u)+\delta^{-1}(u))^2 - (\delta^{-1}(1-u))^2 - (\delta^{-1}(u))^2\right),$$
thus
\begin{align*}
I &=  1 - \alpha(1) - \frac12\int_0^1 (\delta^{-1}(1-u)+\delta^{-1}(u))^2 du + \frac12\int_0^1 (\delta^{-1}(1-u))^2 du + \frac12\int_0^1 (\delta^{-1}(u))^2 du\\
&= 1 - \alpha(1) - \frac12\int_0^1 (\delta^{-1}(1-u)+\delta^{-1}(u))^2 du + \frac12\left(\frac56 - \alpha(1)\right) + \frac12\left(\frac56 - \alpha(1)\right)\\
&= \frac{11}{6} - 2\alpha(1) - \frac12\int_0^1 (\delta^{-1}(1-u)+\delta^{-1}(u))^2 du
\end{align*}
by (e). For the remaining integral we use Jensen's inequality, claiming that
$$ \int_0^1 r(s(x)) dx \ge r\left(\int_0^1 s(x) dx\right), $$
where $s:\II \to A$ is nonnegative measurable function and $r:A \to \RR$ is convex function. So
\begin{align*}
\int_0^1 (\delta^{-1}(1-u)+\delta^{-1}(u))^2 du &\ge \left(\int_0^1 (\delta^{-1}(1-u)+\delta^{-1}(u))du\right)^2 \\
&= \left(1-\alpha(1) +1-\alpha(1)\right)^2 = 4(1-\alpha(1))^2
\end{align*}
by (d), thus
$$ I \le \frac{11}{6} - 2\alpha(1) - 2(1-\alpha(1))^2 = 2\alpha(1) - 2\alpha(1)^2 - \frac16.$$

Finally, to prove (i) we use  \eqref{rho}, (h), and \eqref{phi}
\begin{align*}
\rho(K_\delta) &= 12 \int_0^1 \int_0^1 K_\delta(u,v) dudv -3 \\
&\le 24\alpha(1) - 24\alpha(1)^2 - 5 \\
&= 24\cdot\frac{\phi(K_\delta)+2}{6} - 24\left(\frac{\phi(K_\delta)+2}{6}\right)^2 - 5 \\
&= 1-\frac23 (1-\phi(K_\delta))^2.
\end{align*}
\end{proof}

We can now prove the same estimate for a general copula with arbitrary diagonal.

\begin{theorem}\label{thm:upper}
For any copula $C$ we have
$$ \rho(C) \leq 1-\frac23 (1-\phi(C))^2.$$
\end{theorem}

\begin{proof}
If $C=M$ then $\rho(C)=\phi(C)=1$ and the estimate holds. So assume $C \neq M$ so that $\phi(C)<1$. 
Let $\varepsilon >0$ and $\varepsilon < \frac{1-\phi(C)}{6}$. By \cite[Theorem~4.1.11]{DuSe} there exists an integer $n$ such that the Bernstein copula
$$B_n^C(u,v)=\sum_{i,j=0}^n C\left(\frac{i}{n},\frac{j}{n}\right) \binom{n}{i}\binom{n}{j} u^i(1-u)^{n-i}v^j(1-v)^{n-j}$$
differs from $C(u,v)$ by less than $\varepsilon$ uniformly.
Similarly as in the proof of Theorem~\ref{thm:main} we can take $A^C_n=\frac14 \big(B_n^C+(B_n^C)^t+\widehat{(B_n^C)}+\widehat{(B_n^C)}^t\big)$.
Let $\delta_n(u)$ be the diagonal of copula $A^C_n$. Since the diagonal of $B^C_n$ is a polynomial in $u$, also $\delta_n(u)$ is a polynomial in $u$.
This implies that $\delta_n$ is differentiable with continuous derivative and $\delta_n'(u)=0$ or $\delta_n'(u)=2$ for at most finitely many points $u \in \II$.
Since $A^C_n$ is a doubly symmetric copula, $\delta_n$ is a symmetric diagonal.
Furthermore, it follows from \cite[Theorem~2 (iv)]{NeQuRoUb} that $A^C_n$ is bounded from above by the diagonal copula $K_{\delta_n}$. This together with Proposition~\ref{prop:y} (i) implies that
$$\rho(A^C_n) \le \rho(K_{\delta_n}) \le 1-\frac23 (1-\phi(K_{\delta_n}))^2=1-\frac23 (1-\phi(A^C_n))^2.$$
Similarly as in the proof of Theorem~\ref{thm:main} we estimate
$$
\begin{aligned}
\rho(C) \leq \rho(B^C_n)+12\varepsilon =\rho(A^C_n)+12\varepsilon  \leq 1-\frac23 (1-\phi(A^C_n))^2+12\varepsilon = 1-\frac23 (1-\phi(B^C_n))^2+12\varepsilon.
\end{aligned}
$$
Furthermore, $\phi(B^C_n) \leq \phi(C)+6\varepsilon<1$ by our assumption for $\varepsilon$, hence
$$\rho(C) \leq 1-\frac23 (1-\phi(C)-6\varepsilon)^2+12\varepsilon.$$
By sending $\varepsilon$ to $0$ we obtain the desired estimate.
\end{proof}

Next example shows that for certain values of $\phi(C)$ the upper bound given in Theorem~\ref{thm:upper} is attained.

Let $\{(a_k, b_k), k=1,2, \ldots, n\}$ be a finite family of disjoint open subintervals of $\II$ and $\{B_k, k=1,2, \ldots, n\}$ a family of copulas. Then the \emph{ordinal sum} $B$ of $\{B_k, k=1,2, \ldots, n\}$ with respect to $\{(a_k, b_k), k=1,2, \ldots, n\}$ is a copula defined by 
$$
B(u, v) = \begin{cases}
		a_k + (b_k-a_k)B_k(\frac{u-a_k}{b_k-a_k}, \frac{v-a_k}{b_k-a_k}); & (u, v) \in [a_k, b_k]^2, k=1, 2, \dots, n,\\
			\min\{u,v\}; & \text{otherwise},
    \end{cases}
$$
(see \cite[Section 3.2.2]{Nels}). The Spearman's rho of the ordinal sum $B$ equals
$$ \rho(B) = 1 - \sum_{k=1}^n (b_k-a_k)^3(1-\rho(B_k)). $$

\begin{example} \label{ex:tocke}
Let $n$ be a positive integer and let $C_n$ be a shuffle of $M$
$$ C_n = \textstyle M(2n, (\frac{1}{2n}, \frac{2}{2n}, \dots, \frac{2n-1}{2n}), (2, 1, 4, 3, \dots, 2n, 2n-1), (1, 1, \dots, 1)).$$
The scatterplot and 3D graph of copula $C_3$ is shown in Figure \ref{fig:ex11}.  
\begin{figure}[!ht]
    \centering
    \includegraphics{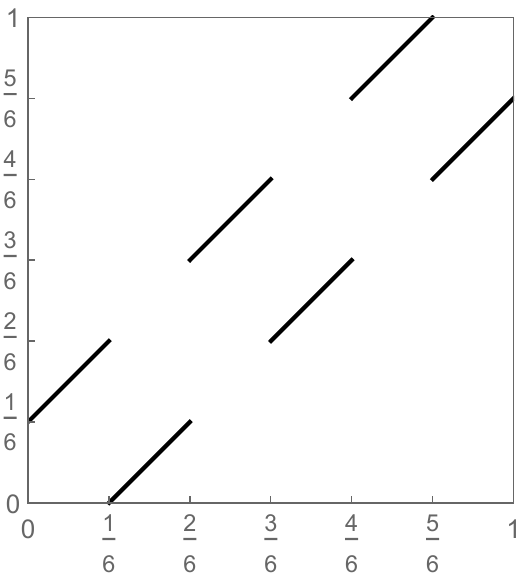} \ \ \ \
    \includegraphics[height=6cm]{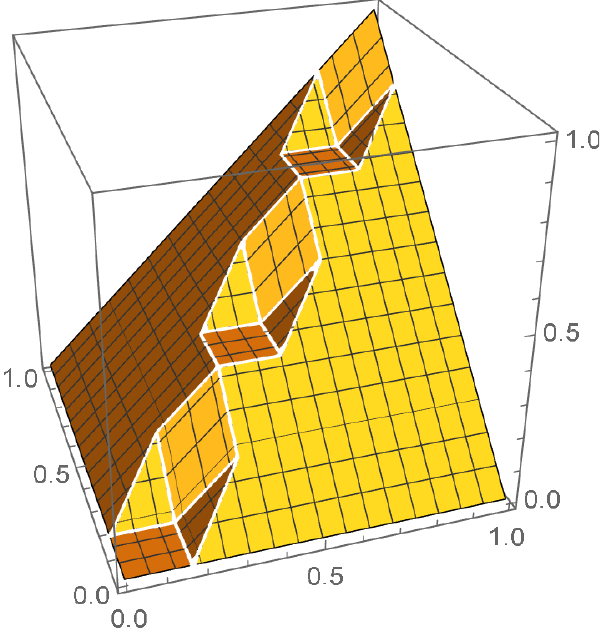}
    \caption{The mass distribution (left) and the 3D graph (right) of the copula $C_3$ from  Example~\ref{ex:tocke}.}
    \label{fig:ex11}
\end{figure}
We have
$$ \delta_{C_n}(u) = \begin{cases}
			\frac{2k-2}{2n}; & u \in [\frac{2k-2}{2n}, \frac{2k-1}{2n}], k=1, 2, \dots, n,\\
			2u-\frac{2k}{2n}; & u \in [\frac{2k-1}{2n}, \frac{2k}{2n}] , k=1, 2, \dots, n,
    \end{cases} $$
so that 
$$ \phi(C_n) = 1 - \frac{3}{2n}.$$
The copula $C_n$ is an ordinal sum of $n$ copies of the copula $C_1$, each of them is of the size $\frac{1}{n}$. Since
$$ C_1(u, v) = \begin{cases}
            0; & 0 \le u \le \frac12, 0 \le v \le \frac12, \\
            u+v-1; & \frac12 \le u \le 1, \frac12 \le v \le 1, \\
            v-\frac12; & 0 \le u \le \frac12, \frac12 \le v \le u + \frac12, \\
            u; & \frac12 \le u \le 1, u + \frac12 \le v \le 1, \\
            u-\frac12; & 0 \le v \le \frac12, \frac12 \le u \le v + \frac12, \\
            v; & 0 \le v \le \frac12, v + \frac12 \le u \le 1, 
            \end{cases}$$
we have $\rho(C_1) = -\frac{1}{2} $ 
and  
$$ \rho(C_n) = 1- n \cdot \frac{1-\rho(C_1)}{n^3} = 1 - \frac{3}{2n^2}.$$
Notice that $\rho(C_n) = 1-\frac23 (1-\phi(C_n))^2$, so the point $(\phi(C_n), \rho(C_n))$ lies on the curve $r=1-\frac23 (1-f)^2$.  \end{example}

In next proposition we show that there is no copula $C$, such that 
$\phi(C) =0$ and the point $(\phi(C), \rho(C))$ lies on the curve $r=1-\frac23 (1-f)^2$. A similar result could be obtained also for some other values of $\phi(C)$.

\begin{proposition}
Suppose that $\phi(C) =0$ for some copula $C \in \CC$. Then $\rho(C) < \frac13$.
\end{proposition}

\begin{proof}
We will actually prove that $\rho(C) \le \frac13 - \frac{12}{1,000,000}$.
Suppose that $\delta$ is a symmetric differentiable diagonal with continuous derivative that satisfies \eqref{eq:delta-smooth}, $\alpha, g, h:\II \to \II$ auxiliary functions defined by \eqref{eq:aux}, and $K_\delta$ diagonal copula.
Furthermore suppose that $\phi(K_\delta) = 0$, so that $\alpha(1) = \frac13$.
Let $\varepsilon = \frac{1}{100}$ and denote by $b = \delta(\frac13 - \varepsilon)$. We will first show that $b \ge \frac{1}{50}$. To this end we may assume that $b\le\frac16-2\varepsilon$.

Since $\delta$ is increasing, we have $\delta(u) \le b$ for any $u \le \frac13 - \varepsilon$, and since it is 2-Lipshitz we have $\delta(u) \le 2u- \frac23 + 2\varepsilon + b$ for any $u \ge \frac13 - \varepsilon$. Since $\delta$ is a symmetric diagonal we have $\delta(\frac23 + \varepsilon) = b+\frac13+ 2\varepsilon$.
We can derive similar estimates as above using the point $\frac23+\varepsilon$. It follows that for any $u \in \II$
$$ \delta(u) \le \delta_1(u) = \min\bigl\{u, \max\{b, 2u- \frac23 + 2\varepsilon + b\},\max\{b+\frac13+ 2\varepsilon, 2u+b-1\}\big\} . $$
Since $b\leq \frac16-2\varepsilon$, we have
$$ \alpha(1) = \frac 13 \le \int_0^1 \delta_1(u)du = \frac{1}{36}(11 + 36b -36b^2 + 24\varepsilon + 72\varepsilon^2),$$
hence 
$$ b \ge \frac12 - \frac{1}{3}\sqrt{2 + 6\varepsilon + 18\varepsilon^2} \approx 0.0214 \ge \frac{1}{50}.$$ 

For any $u\le \frac{1}{50}$ we now have
$$ \textstyle\delta^{-1}(u) \le \delta^{-1}(\frac{1}{50}) \le \delta^{-1}(b) = \frac13 - \varepsilon, $$
 and since $\delta^{-1}(1-u) \le 1$ it follows that 
$$ \textstyle\frac43 - \delta^{-1}(u) - \delta^{-1}(1-u) \ge \varepsilon. $$
Thus 
$$ \int_0^1 \left(\frac43 - \delta^{-1}(u) - \delta^{-1}(1-u)\right)^2 du \ge \int_0^{1/50} \left(\frac43 - \delta^{-1}(u) - \delta^{-1}(1-u)\right)^2 du \ge \frac{1}{50}\varepsilon^2 .$$
It follows that
\begin{align*}
 \int_0^1 \left(\delta^{-1}(u) + \delta^{-1}(1-u)\right)^2 du &= \int_0^1 \left(\frac43 - \delta^{-1}(u) - \delta^{-1}(1-u)\right)^2 du \\
 & \ \ \ \ \ + \frac83\int_0^1 \left(\delta^{-1}(u) + \delta^{-1}(1-u)\right) du - \frac{16}{9} \\
 &\ge \frac{1}{50}\varepsilon^2 + \frac{16}{3}(1-\alpha(1)) - \frac{16}{9} \\
 &= \frac{1}{50}\varepsilon^2 + \frac{16}{9} 
\end{align*}
by Proposition \ref{prop:y} (d). It was proven in Proposition \ref{prop:y} (h) that
$$ I = \int_0^1 \int_0^1 K_\delta(u,v) dudv = \frac{11}{6} - 2\alpha(1) - \frac12\int_0^1 (\delta^{-1}(1-u)+\delta^{-1}(u))^2 du ,$$
so we have
$$ I \le \frac{11}{6} - \frac23 - \frac12\left(\frac{1}{50}\varepsilon^2 + \frac{16}{9}\right) = \frac{5}{18} - \frac{1}{100}\varepsilon^2 .$$
It follows that 
$$ \rho(K_\delta) = 12 I -3 \le \frac13 - \frac{12}{100}\varepsilon^2 = \frac13 -  \frac{12}{1,\!000,\!000}. $$

Finally, a similar argument as in the proof of Theorem \ref{thm:upper} shows that for any copula with $\phi(C) = 0$ we have $\rho(C) \le \frac13 -  \frac{12}{1,000,000}. $
\end{proof}

Nevertheless, for any value of $\phi(C)$ we can come close to the upper bound proved in Theorem~\ref{thm:upper} as the next example and proposition demonstrate.

\begin{example} \label{ex:10}
Let $a \in [\frac14, \frac12]$ and let $\delta_a$ be a diagonal
$$ \delta_a(u) = \begin{cases}
			0; & u \le a, \\
			u-a; & a \le u \le 1-a, \\
			2u-1; & 1-a \le u \le 1.
    \end{cases} $$
The diagonal copula belonging to $\delta_a$ is
$$ K_{\delta_a}(u, v) = \begin{cases}
            0; & 0 \le u \le a, 0 \le v \le a, \\
            \frac{v-a}{2}; & a \le v \le 1-a, \frac{v-a}{2} \le u \le a, \\
            \frac{u-a}{2}; & a \le u \le 1-a, \frac{u-a}{2} \le v \le a, \\
            v-\frac12; & \frac12-a \le u \le a, 1-a \le v \le u+\frac12, \\
            \frac{u+v}{2}-a; & a \le u \le 1-a, a \le v \le 1-a, \\
            u-\frac12; & \frac12-a \le v \le a, 1-a \le u \le v+\frac12, \\
            \frac{u+2v-a-1}{2}; & a \le u \le 1-a, 1-a \le v \le \frac{u+a+1}{2}, \\
            \frac{2u+v-a-1}{2}; & a \le v \le 1-a, 1-a \le u \le \frac{v+a+1}{2}, \\
            u+v-1; & 1-a \le u \le 1, 1-a \le v \le 1, \\
            u; & \min\{2u+a, u+\frac12, \frac{u+a+1}{2}\} \le v \le 1 \\
            v; & \min\{2v+a, v+\frac12, \frac{v+a+1}{2}\} \le u \le 1. 
            \end{cases}$$
The scatterplot and the 3D graph of copula $K_{\delta_a}$ is shown in Figure \ref{fig:ex10}. 
\begin{figure}[!ht]
    \centering
    \includegraphics{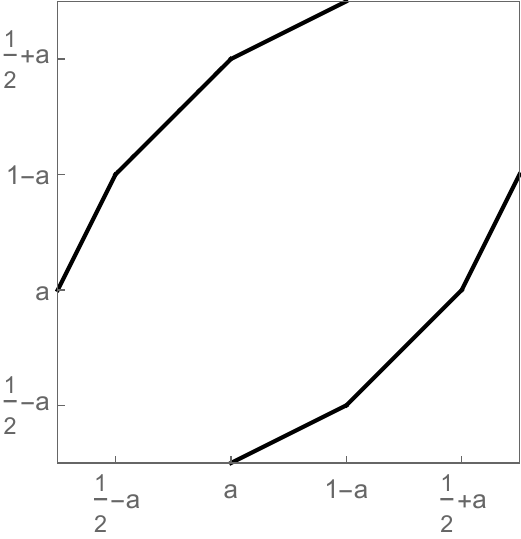} \ \ \ \
    \includegraphics[height=6cm]{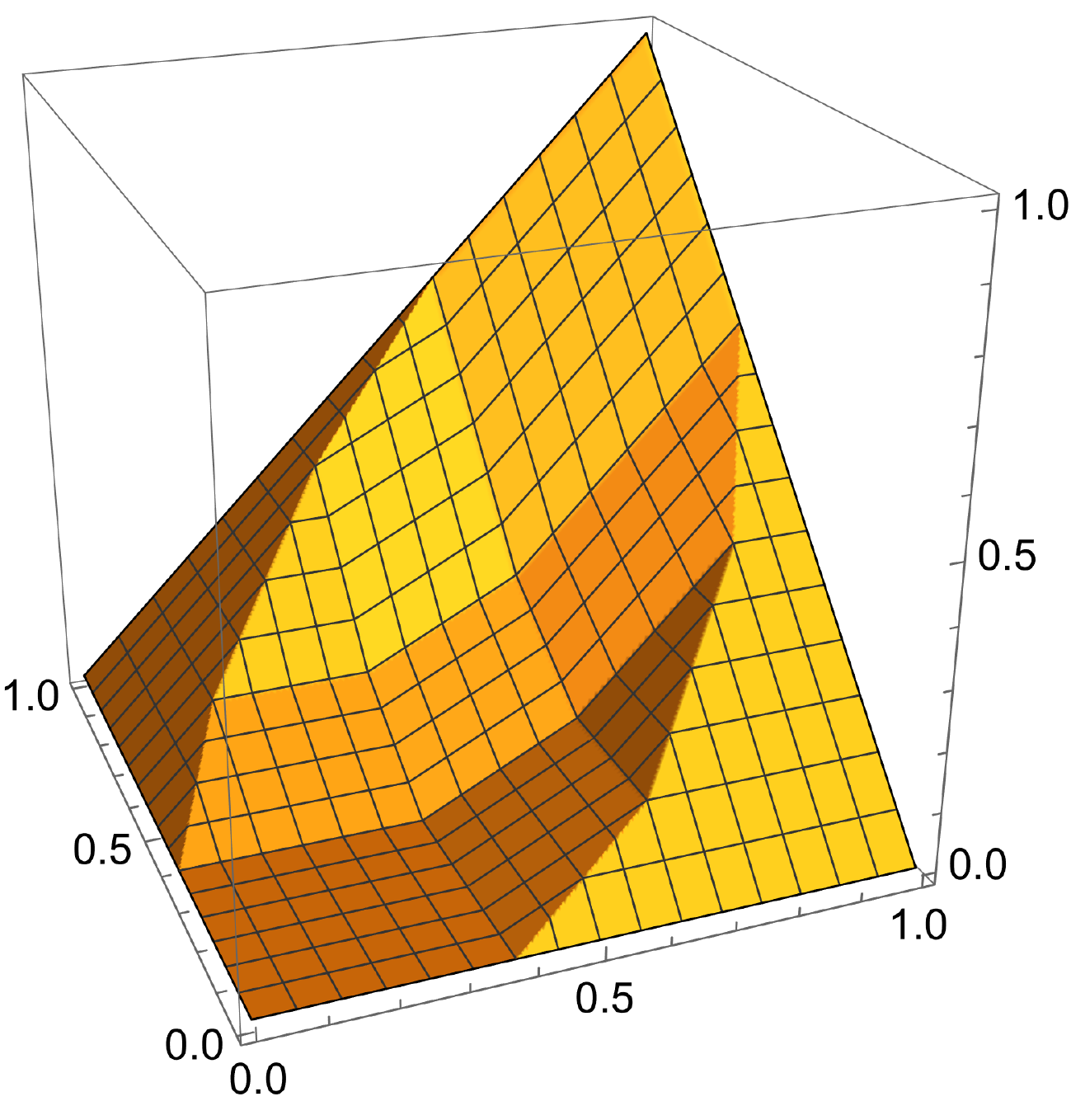}
    \caption{The mass distribution (left) and the 3D graph (right) of the copula $K_{\delta_a}$ from Example~\ref{ex:10}.}
    \label{fig:ex10}
\end{figure}

It follows that
$$ \phi(K_{\delta_a}) = 6a^2 - 6a + 1  
\text{ \ \ and \ \ }  \rho(K_{\delta_a}) = 8a^3 - 6a + \tfrac32,$$
so that 
$$\rho(K_{\delta_a}) = - \frac12 + (1 + 2\phi(K_{\delta_a})) -  \frac{\sqrt{3}}{9}(1+2\phi(K_{\delta_a}))^{3/2}.$$ 
Note that $\phi(K_{\delta_a}) \in [-\frac12,-\frac18]$ for $a \in [\frac14, \frac12]$. The point $(\phi(K_{\delta_a}), \rho(K_{\delta_a}))$ lies strictly below the curve $r=1-\frac23 (1-f)^2$ for any $a \in [\frac14, \frac12)$. 
\end{example}

Let $r: [-\frac12, 1] \to [-1, 1]$ be a function defined by
\begin{equation}\label{eq:r}
    r(x) =  \begin{cases}
            2x + \frac12 -  \frac{\sqrt{3}}{9}(1+2x)^{3/2}; & x \in [-\frac12, -\frac18], \\
            \frac43 x + \frac{7}{24}; & x \in [-\frac18, \frac14], \\
            {\displaystyle\frac{2n+1}{n^2+n}x + \frac{2n^2-2n+1}{2(n^2+n)}}; & x \in [1-\frac{3}{2n}, 1-\frac{3}{2(n+1)}] \text{ for } n= 2, 3, \dots, \\
            1; & x = 1. 
            \end{cases}
\end{equation}

\begin{proposition}\label{prop:dosezeno}
For any point $(x, r(x))$ on the graph of function $r$ there exist a copula $C$, such that
$\phi(C) = x$ and $\rho(C) = r(x)$.
\end{proposition}

\begin{proof}
Example \ref{ex:10} shows that for $x \in [-\frac12, -\frac18]$ any point on the graph  of function $r$  is attained by copula $K_{\delta_a}$. Note that $r(1-\frac{3}{2n}) = 1-\frac{3}{2n^2}$, so for $x = 1-\frac{3}{2n}$ the point on the graph  of function $r$ is attained by copula $C_n$ from Example \ref{ex:tocke}. In the interval $[-\frac18, \frac14]$ the points on the graph  of function $r$ are attained by  convex combinations of copulas $K_{\delta_{1/4}}$ and $C_2$ and in the intervals $[1-\frac{3}{2n}, 1-\frac{3}{2(n+1)}]$ by  convex combinations of copulas $C_n$ and $C_{n+1}$.
\end{proof}

\section{The exact region determined by $\phi$ and $\rho$}\label{sec:region}

We can now collect our findings in the following theorem.

\begin{theorem}
The exact region determined by Spearman's rho and Spearman's footrule of all points 
$\{(\phi(C), \rho(C)) \in [-\frac12, 1] \times [-1,1]; C \in \CC\}$ is given by

$$ \frac29\sqrt{3}(1+2\phi(C))^{3/2}-1 \le \rho(C) \le s(\phi(C))$$
where $s: [-\frac12, 1] \to [-1, 1]$ is a concave function satisfying
$$ r(x) \le s(x) \le 1-\frac23 (1-x)^2$$
and $r$ is the function defined by \eqref{eq:r}. 
\end{theorem}

\begin{proof}
The assertion follows directly from Theorem \ref{thm:main}, Theorem \ref{thm:upper}, and Proposition \ref{prop:dosezeno}. The function $s$ is concave since the exact region is convex. 
\end{proof}

Note that the role of $\phi$ and $\rho$ can be exchanged, so from the theorem one can derive the exact upper bound for $\phi(C)$ in terms of $\rho(C)$ and a tight estimate for the lower bound.

Figure \ref{fig:fi-ro} shows the exact region determined by Spearman's rho and Spearman's footrule. The graph of function $r(x)$ is shown full, and the graph of function $1-\frac23 (1-x)^2$ is dashed.

\begin{figure}[!ht]
    \centering
    \includegraphics[height=8cm]{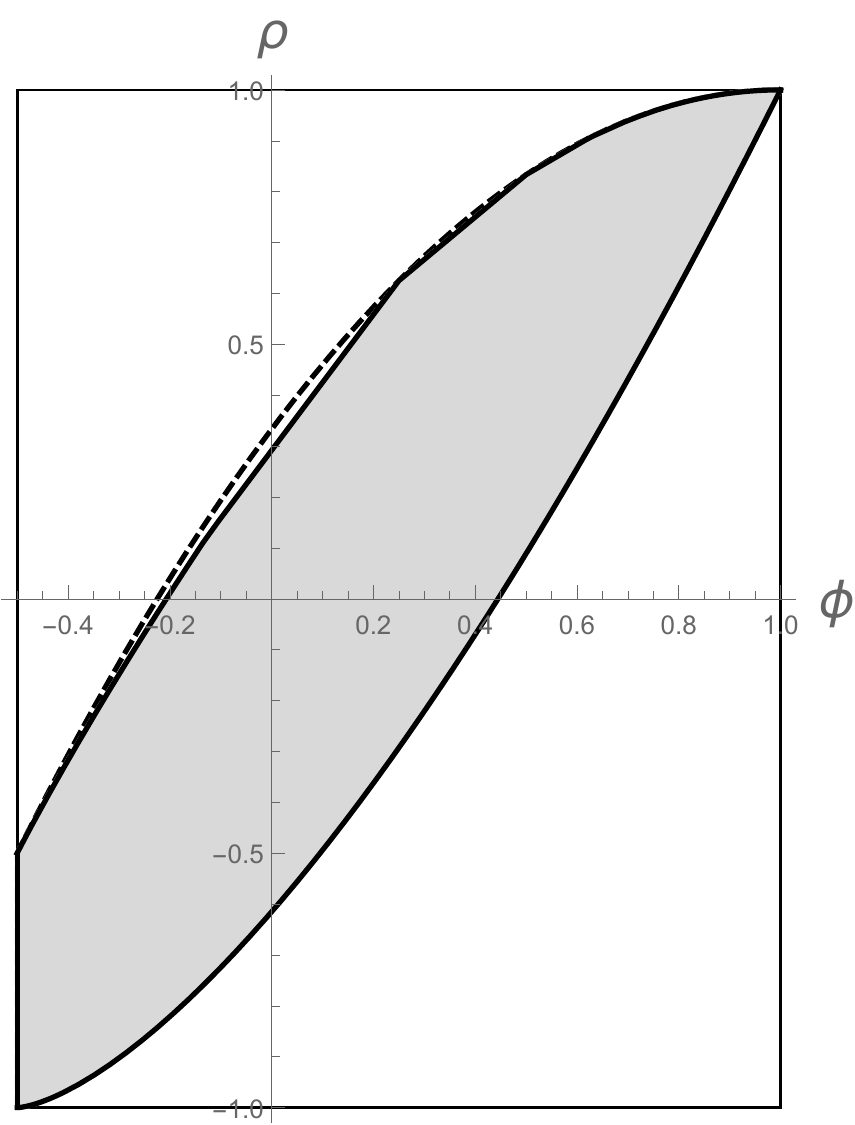} 
    \caption{The exact region determined by Spearman's rho and Spearman's footrule.}
    \label{fig:fi-ro}
\end{figure}

In paper \cite{KoBuMo} the authors introduce the \emph{$(\kappa_1, \kappa_2)$-similarity measure} between (week) concordance measures $\kappa_1$ and $\kappa_2$ as
$$ \kappa sm(\kappa_1, \kappa_2) = 1 - \frac{A(\kappa_1, \kappa_2)}{(1-\kappa_1(W))(1-\kappa_2(W))},$$ 
where $A(\kappa_1, \kappa_2)$ is the area of the exact region determined by $\kappa_1$ and $\kappa_2$. Our last proposition estimate $\kappa sm(\phi, \rho)$.

\begin{proposition}\label{prop:ksm}
The $(\phi,\rho)$-similarity measure between Spearman's footrule and Spearman's rho  satisfies
$$ 0.65 = \frac{13}{20} \le \kappa sm(\phi, \rho) \le \frac{121}{64} - \frac{\pi^2}{8} \approx 0.6569 .$$
\end{proposition}

\begin{proof}
We have 
$$ \kappa sm(\phi, \rho) = 1 - \frac{A(\phi, \rho)}{3} $$
from the definition and 
$$ A(\phi, \rho) \le \int_{-1/2}^1 \left(1-\frac23 (1-x)^2\right)dx - \int_{-1/2}^1 \left(\frac29\sqrt{3}(1+2x)^{3/2}-1\right)dx = \frac{3}{4} + \frac{3}{10} = \frac{21}{20} ,$$
so the lower bound follows. On the other hand, 
\begin{align*}
A(\phi, \rho) &\ge \int_{-1/2}^1 r(x)dx - \int_{-1/2}^1 \left(\frac29\sqrt{3}(1+2x)^{3/2}-1\right)dx \\
&= \int_{-1/2}^{-1/8} \left(2x + \frac12 -  \frac{\sqrt{3}}{9}(1+2x)^{3/2}\right)dx + \frac{9}{64} \\
& \ \ \ \ + \sum_{n=2}^{\infty} \frac12\left(1-\frac{3}{2n^2} + 1-\frac{3}{2(n+1)^2}\right)\left(\frac{3}{2n}-\frac{3}{2(n+1)}\right) + \frac{3}{10} \\
&= - \frac{21}{320} + \frac{9}{64} + \frac{3\pi^2}{8} - \frac{195}{64} + \frac{3}{10} = \frac{3\pi^2}{8} -\frac{171}{64}
\end{align*}
and the upper bound follows.
\end{proof}

Note that the exact region determined by Spearman's footrule and Spearmsn's rho is similar in shape to the exact region determined by Sperman's rho and Kendall's tau, i.e., the upper bound seems to be a piecewise function with finer and finer pieces.
However, the exact region determined by Sperman's rho and Kendall's tau is not convex while in our case the region is convex.

\bibliographystyle{amsplain}
\bibliography{biblio}

\end{document}